\documentclass[11pt,a4paper]{amsart}
\usepackage{amssymb, amsmath, amscd, mathrsfs,marvosym, psfrag,color} %, txfonts} 

\input xy
\xyoption{all}
\DeclareMathAlphabet{\mathpzc}{OT1}{pzc}{m}{it}

\newtheorem{theorem}{Theorem}[section]

\newtheorem*{theoremA}{Theorem A}
\newtheorem*{theoremB}{Theorem B}
\newtheorem*{theoremC}{Theorem C}
\newtheorem*{theoremD}{Theorem D}

\newtheorem{proposition}[theorem]{Proposition}

\newtheorem{lemma}[theorem]{Lemma}

\theoremstyle{definition}
\newtheorem{definition}[theorem]{Definition}

\theoremstyle{remark}
\newtheorem{remark}[theorem]{Remark}

\newcommand{\CX}{{\mathcal X}}

\newcommand{\SZ}{{\mathscr Z}}

\newcommand{\DZ}{{\mathbb Z}}

\newcommand{\ch}{{\operatorname{char}\, }}

\newcommand{\Frob}{{\operatorname{Frob}}}

\newcommand{\End}{{\operatorname{End}}}

\newcommand{\Hom}{{\operatorname{Hom}}}

\newcommand{\Extend}{{\operatorname{\mathsf{E}}}}

\newcommand{\Res}{{\operatorname{\mathsf{R}}}}

\newcommand{\im}{{\operatorname{im}\,}}

\newcommand{\ol}{\overline}

\newcommand{\id}{{\operatorname{id}}}

\newcommand{\comment}[1]{}

\newcommand{\lgl}{\langle}
\newcommand{\rgl}{\rangle}

\newcommand{\qchoose}{\genfrac{[}{]}{0pt}{}}

\begin{document}

\pagenumbering{arabic}
\title[]{Representations and binomial coefficients} \author[]{Peter Fiebig}
\begin{abstract} To a root system $R$ and a {\em choice of coefficients} in a field $K$ we associate a category $\CX$ of {\em graded spaces with operators}. %Objects in this category are vector spaces that carry a grading by the weight lattice and are endowed with linear operators in simple root directions. The coefficients determine commutation relations between the operators. 
For an arbitrary choice of coefficients we show that  we obtain a semisimple category in which the simple objects are  parametrized by their highest weight. Then we assume that  the coefficients are given by quantum binomials associated to $(K,q)$, where $q$ is an invertible element in $K$.  In the case that $R$ is simply laced and $(K,q)$ has positive (quantum) characteristic, we construct a Frobenius pull-back functor and prove a version of Steinberg's tensor product theorem for $\CX$.  Then we  prove that one can view the objects in $\CX$  as the  semisimple representations of Lusztig's quantum group associated to $(R,K,q)$ (for $q=1$ we obtain semisimple representations of the hyperalgebra associated to $(R,K)$). Hence we obtain new proofs of the Frobenius and Steinberg theorems both in the modular and the root of unity cases.   \end{abstract}

\address{Department Mathematik, FAU Erlangen--N\"urnberg, Cauerstra\ss e 11, 91058 Erlangen}
\email{fiebig@math.fau.de}
\maketitle
%\tableofcontents
\section{Introduction}
Let $R$ be a root system with basis $\Pi$ and weight lattice $X$. Let   $K$  be a field and let $c$ be a map  that associates an element in $K$ to a tuple $(\mu,\alpha,m,n,r)$, where $\mu$ is a weight in $X$, $\alpha$ is a simple root in $\Pi$, and $m$, $n$ and $r$ are non-negative integers. To these choices we associate a category $\CX$ of {\em graded spaces with operators}. Objects in this category are  $X$-graded $K$-vector spaces  $M=\bigoplus_{\mu\in X} M_\mu$, endowed with linear operators $E_{\alpha,n}$ and $F_{\alpha,n}$ for $\alpha\in\Pi$ and $n>0$ of degree $+n\alpha$ and $-n\alpha$, resp., subject to the following axioms. 

\begin{itemize}
\item Each graded subspace $M_\mu$ is finite dimensional and the set of $\mu$ with $M_\mu\ne 0$ is bounded from above.
\item The operators $E_{\alpha,m}$ and $F_{\beta,n}$ commute if $\alpha\ne\beta$, and 
$$
E_{\alpha,m} F_{\alpha,n}|_{M_{\mu+n\alpha}}=\sum_r c(\mu,\alpha,m,n,r)F_{\alpha,n-r}E_{\alpha,m-r}|_{M_{\mu+n\alpha}}.
$$ 
%restricted to a weight space $M_\mu$, is the $K$-linear combination of operators of the form $F_{\alpha,n-r}E_{\alpha,m-r}$ with coefficients $c(\mu,\alpha,m,n,r)$. 
\item  Each weight space $M_\mu$ is the direct sum of its primitive vectors and its coprimitive vectors.
\end{itemize}
(A vector in $M_\mu$ is {\em primitive} if it is annihilated by all $E$-operators, and {\em coprimitive} if it is  contained in the subspace generated by the images of the $F$-operators.) The above are inspired by similar axioms that appear in \cite{LefOp,TiltLoc}. In particular, they can be considered as defining {\em Lefschetz operators in multiple simple root directions}.

Our first result is that $\CX$ is a semisimple category with simple objects being parametrized by the weight lattice $X$.
\begin{theoremA} 
 \begin{enumerate}
\item For all $\lambda\in X$ there is an up to isomorphism unique object $S(\lambda)$ in $\CX$ with the following properties.
\begin{enumerate}
\item $S(\lambda)$ is indecomposable.
\item $S(\lambda)_\lambda$ is a one-dimensional vector space, and $S(\lambda)_\mu\ne 0$ implies $\mu\le\lambda$.
\end{enumerate}
\item The objects $S(\lambda)$ characterized in (1) satisfy 
$$
\Hom_\CX(S(\lambda),S(\mu))=\begin{cases}
0,&\text{ if $\mu\ne\lambda$}, \\
K\cdot \id_{S(\lambda)},&\text{ if $\mu=\lambda$}.
\end{cases}
$$
\item For any object $M$ in $\CX$ there exists an index set $J$ and weights $\lambda_j\in X$ for $j\in J$ such that $M\cong \bigoplus_{j\in J}S(\lambda_j)$. The multiset of weights $\{\lambda_j\}$ is uniquely determined by $M$.
\end{enumerate}
\end{theoremA}
The character of $S(\lambda)$ highly depends on the choice of coefficients $c$, and we cannot say much about the characters in  this generality. 
%This is Theorem \ref{thm-constrS}.
But we  show that each object in $\CX$ carries a non-degenerate {\em contravariant form}, i.e.~a symmetric bilinear form with the properties that the weight decomposition is orthogonal and the $E$-operators are adjoint to the corresponding $F$-operators. For this we have to assume that the choice of coefficients $c$ is symmetric in $m$ and $n$.

So far, our results hold for an (almost) arbitrary choice of coefficients  $c$. In the second part of the paper we assume that $R$ is simply laced and that $c$ is given by certain {\em binomial coefficients}, or, more generally, {\em quantum binomial coefficients}. Let $q$ be an invertible element in $K$. Then one defines for any $n\in\DZ$ the (normalized) {\em quantum integer}  $[n]=\frac{q^n-q^{-n}}{q-q^{-1}}\in K$, and for  $a,b\in\DZ$ the {\em quantum binomial coefficient} $\qchoose ab\in K$ (it is convenient to set $\qchoose ab=0$ for $b<0$). %We fix an invertible element $q$  in  $K$, and denote by the same symbols  $[n]$ and $\qchoose ab$ the elements in $K$ that we obtain  by specializing $v$ to $q$. 
Our choice of coefficients now is
$$
c(\mu,\alpha,m,n,r)=\qchoose{\lgl\mu,\alpha^\vee\rgl+m+n}{r}.
$$
In order to stress the dependence on $q$ we denote the resulting category sometimes by $\CX_{(K,q)}$ and its simple  objects by $S_{(K,q)}(\lambda)$. 

We  deduce several properties of the $E$- and $F$-operators, and of the objects $S(\lambda)$, from arithmetic properties of the (quantum) binomial coefficients. 
For example, we obtain that 
$$
E_{\alpha,m}E_{\alpha,n}=\qchoose{m+n}{m}E_{\alpha,m+n}\text{ and }F_{\alpha,m}F_{\alpha,n}=\qchoose{m+n}{m}F_{\alpha,m+n}
$$
for all $\alpha\in\Pi$ and $m,n\ge0$. We also deduce the Serre relations and some of their higher analogues that were proven by Lusztig: If $\lgl\alpha,\beta^\vee\rgl=-1$ and $m\ge 2$, then 
\begin{align*}
\sum_{r} (-1)^r q^{ r(2-m)}F_{\alpha,r}F_{\beta,1}F_{\alpha,m-r} &=0,\\
\sum_{r} (-1)^r q^{r(2-m)}E_{\alpha,r}E_{\beta,1}E_{\alpha,m-r} &=0.
\end{align*}
It is worthwhile to point out here that the binomial identity used to prove the above results is not one of the most basic ones. It is a  $q$-version of the famous {\em Pfaff-Saalsch\"utz identity} that was discovered in the theory of hypergeometric functions. It reads 
$$
\qchoose{x+a}{a}\qchoose{y+b}{b}=\sum_{k} \qchoose{x+y+k}{k}\qchoose{x+a-b}{a-k}\qchoose{y+b-a}{b-k}
$$
 for all $a,b,x,y\in\DZ$.

To the pair $(K,q)$ we can associate its {\em quantum characteristic $\ell$} as follows. We set $\ell=0$ if $[n]\ne 0$ for all $n>0$. Otherwise we let $\ell$ be the smallest positive integer such that $[\ell]=0$. If $\ell>0$ , then $q$ is a root of unity in $K$, and for $q=\pm 1$ the quantum characteristic coincides with the (ordinary) characteristic of the field $K$. If $\ell>0$, then one has further relations among the (quantum) binomial coefficients.  For example,   for  $a,b\in\DZ$ we have 
$$
\qchoose{\ell a}b=
\begin{cases} 0,&\text{ if $b\not\in\ell\DZ$},\\
{a\choose b/\ell},&\text{ if $b\in\ell\DZ$}.
\end{cases}
$$
Note that $a\choose b/\ell$ denotes the {\em ordinary} binomial coefficient, i.e.~the one we obtain in the case $q=1$. The above relation  is the main ingredient in the construction of the {\em  Frobenius pull-back functor}.
\begin{theoremB} Suppose that $\ell>0$ and that the order of $q$ is odd if $q\ne \pm1$. There exists a functor $\Frob^{\ast}\colon \CX_{(K,1)}\to\CX_{(K,q)}$ with the property 
$$
\Frob^{\ast}(S_{(K,1)}(\lambda))\cong S_{(K,q)}(\ell\lambda).
$$
\end{theoremB} 
%This is Theorem \ref{thm-Frob}.

We also employ the $q$-version of {\em Lucas' theorem}, i.e.
$$
\qchoose ab=\qchoose{a_0}{b_0}{a_1\choose b_1}
$$
for all $a,b\in\DZ$  with  $a=a_0+\ell a_1$, $b=b_0+\ell b_1$ and $0\le a_0,b_0<\ell$,  and a version of the {\em $q$-Chu-Vandermonde convolution} formula
$$
\qchoose{a+\ell b}{n}=\sum_{n=r+\ell s} \qchoose{a}{r}{ b\choose s}
$$ 
for all $a,b,n\in\DZ$. These identites are used to prove the following.
\begin{theoremC} Let $\lambda_0,\lambda_1\in X$ and suppose that $\lambda_0$ is {\em restricted}, i.e.~$0\le\lgl\lambda_0,\alpha^\vee\rgl<\ell$ for all $\alpha\in\Pi$. Then there exist $E$- and $F$-operators on the $X$-graded space $S(\lambda_0)\otimes S(\ell\lambda_1)$ such that we obtain an object in $\CX$ with
$$
S(\lambda_0)\otimes S(\ell\lambda_1)\cong S(\lambda_0+\ell\lambda_1).
$$
\end{theoremC}
%This is Theorem \ref{thm-Steinb} in this article. 
Finally, we show that our category $\CX$ and its objects $S(\lambda)$ have a real world interpretation. For this, let $U_K$ be the quantum group (with divided powers) associated to $R$ and the pair $(K,q)$. It is an associative unital $K$-algebra generated by elements $e_{\alpha}^{[n]}$,  $f_{\alpha}^{[n]}$, $K_\alpha^{\pm1}$ for $\alpha\in\Pi$ and $n>0$ and some relations, cf.~\cite{L90}. For all $\lambda\in X$ there exists an up to isomorphism unique  simple $U_K$-module $L(\lambda)$ of highest weight $\lambda$. Note that if $q=1$, then the action of $U_K$ on $L(\lambda)$ factors over a quotient that is isomorphic to the hyperalgebra of the semisimple, simply connected algebraic group $G_K$ over $K$ associated with $R$. If $\lambda$ is dominant, then $L(\lambda)$ is the irreducible representation of $G_K$ with highest weight $\lambda$.

\begin{theoremD} Let $\lambda\in X$. Then $L(\lambda)$, together with its weight decomposition and the homomorphisms $E_{\alpha,n}$ and $F_{\alpha,n}$ coming from the action of the standard generators $e_{\alpha}^{[n]}$ and $f_{\alpha}^{[n]}$, is an object in $\CX$. It is isomorphic to $S(\lambda)$.
\end{theoremD}
Once this is established, Theorem B and Theorem C yield the Frobenius pull-back and Steinberg's tensor product theorem for simple $U_K$-modules, resp.

The research that led to  this article was motivated by the generational phenomenon observed by Lusztig and Lusztig--Williamson for characters of simple and tilting modules for algebraic groups in positive characteristics (cf. \cite{L15,LW}). This phenomenon is strongly linked to both the Frobenius pull-back and Steinberg's tensor product theorem. The results of this article show that both statements can be traced back to certain arithmetic properties of binomial coefficients.

\subsection*{Notational conventions} 
For convenience we use the following conventions:
\begin{align*}
\sum_r&=\sum_{r\in\DZ},\\
\sum_{r+s=n}&=\sum_{(r,s)\in\DZ^2\atop r+s=n} \text{ (with fixed $n$)}.
\end{align*}
Although the above summations are infinite, in each case in this paper only a finite number of summands will be non-zero. Since it is not necessary to keep track of the summation boundaries, change of variable arguments become much simpler. 
In the second half of the paper many binomial coefficients will appear. It is very convenient (and customary) to interpret ${a\choose b}=0$ and $\qchoose ab=0$ for all $b<0$. 

\subsection*{Acknowledgement} The author would like to thank Jens Carsten Jantzen for a remark %(long before this article was written) 
that simplified the construction of the extension functor in Section \ref{subsec-ext}  significantly.

\section{The category of graded spaces with operators} \label{sec-gradspaces} 
The purpose of this section is to define the category $\CX$ and show that it is semisimple with simple objects being parametrized by their highest weights. One of the advantages of the category $\CX$ is that its objects, and even the category itself, can be approximated, i.e.~there are versions $\CX_I$ of $\CX$, where $I$ is an upwardly closed subset of the set of weights $X$. This allows us to construct the simple objects in $\CX$ ``weight by weight''.

\subsection{Setup}

We fix a root system $R$ and a basis $\Pi$ of $R$. For any $\alpha\in R$ we denote by $\alpha^\vee\in R^\vee$ its coroot. We let $X$ be the weight lattice of $R$ and $\le$  the usual partial order on $X$ with respect to $\Pi$, i.e.~$\mu\le\lambda$ if and only if $\lambda-\mu$ can be written as a sum of elements of $\Pi$.

\begin{definition}\begin{enumerate}
\item A subset $T$ of $X$ is called  {\em quasi-bounded (from above)} if for all $\mu\in X$ the set $\{\lambda\in T\mid \mu\le\lambda\}$ is finite. 
\item A subset $I$ of $X$ is called {\em closed} if $\lambda\in I$ and $\lambda\le\mu$ imply $\mu\in I$.
\end{enumerate}
\end{definition}

Let $K$ be a field.

\begin{definition} A {\em choice of coefficients} is a map
\begin{align*}
c\colon X\times \Pi\times \DZ_{\ge0}^2\times\DZ&\to K,\\ 
(\mu,\alpha,m,n,r)&\mapsto c_{\mu,\alpha,m,n,r}
\end{align*}
with the properties 
$$
c_{\mu,\alpha,m,n,0}=1\text{ and } c_{\mu,\alpha,m,n,r}=0 \text{ for $r<0$}
$$
for all $\mu,\alpha,m,n$.
\end{definition}

One of the main examples that we are interested in is the case $c_{\mu,\alpha,m,n,r}={\lgl\mu,\alpha^\vee\rgl+m+n\choose r}$ and its quantum version (cf.~Section \ref{sec-GbC}).

\subsection{Graded spaces with operators}
Let $I$ be a closed subset of $X$ and $M=\bigoplus_{\mu\in I}M_\mu$ an $I$-graded $K$-vector space. We call $\mu\in I$ a {\em weight} of $M$ if $M_\mu\ne 0$. We assume the following.
\begin{enumerate}
\item[(X1)] The set of weights of $M$  is quasi-bounded from above and each $M_\mu$ is a finite dimensional $K$-vector space. 
\end{enumerate}
Now suppose that $M$ is endowed with homogeneous linear operators 
\begin{align*}
E_{\mu,\alpha,n}&\colon M_\mu\to M_{\mu+n\alpha},\\
F_{\mu,\alpha,n}&\colon M_{\mu+n\alpha}\to M_\mu
\end{align*} 
for all $\mu\in I$, $\alpha\in\Pi$ and $n>0$ (note that since $I$ is closed, $\mu\in I$ implies $\mu+n\alpha\in I$ for all $\alpha\in\Pi$ and $n>0$). It is convenient to set $E_{\mu,\alpha,0}=F_{\mu,\alpha,0}=\id_{M_\mu}$ and $E_{\mu,\alpha,n}=0$, $F_{\mu,\alpha,n}=0$ for all $n<0$.  For notational simplicity we  often write $E_{\alpha,n}$ and $F_{\alpha,n}$ instead of  $E_{\mu,\alpha,n}$ and $F_{\mu,\alpha,n}$ if the weight $\mu$ of the argument is clear from the context. Sometimes we write $F^M_{\alpha,n}$, $E_{\alpha,n}^M$, etc. if we want to specify on which graded space the operators act. We assume that these operators satisfy the following axiom.
\begin{enumerate}
\item[(X2)] For all  $\mu\in I$, $\alpha,\beta\in\Pi$, $m,n>0$ and $v\in M_{\mu+n\beta}$ we have
$$
E_{\alpha,m}F_{\beta,n}(v)=\begin{cases}
F_{\beta,n}E_{\alpha,m}(v),&\text{ if $\alpha\ne\beta$},\\
\sum_rc_{\mu,\alpha,m,n,r}F_{\alpha,n-r}E_{\alpha,m-r}(v),&\text{ if $\alpha=\beta$}.
\end{cases}
$$
\end{enumerate}
Note that due to our assumption that $c_{\mu,\alpha,m,n,r}=0$ for $r<0$, only summands of the form $F_{\alpha,s}E_{\alpha,t}(v)$ with $s\le n$ and $t\le m$ appear in the above summation. Since we assume that $I$ is closed, all the operators on the right are well defined.

In order to formulate the third and last axiom for our data, we define 
for $\mu\in I$ the direct summand $M_{\delta\mu}:=\bigoplus_{\alpha\in \Pi,n>0} M_{\mu+n\alpha}$ of $M$. Note that the axiom (X1) implies that only finitely many summands of $M_{\delta\mu}$ are non-zero. Hence we can define
\begin{align*}
E_\mu&\colon M_\mu\to M_{\delta\mu},\\
F_\mu&\colon M_{\delta\mu}\to M_\mu
\end{align*}
as the column and the row vector, resp., with entries $E_{\mu,\alpha,n}$ and $F_{\mu,\alpha,n}$, resp. More explicitely, $F_\mu((v_{\mu+n\alpha})_{\alpha,n})=\sum_{\alpha\in\Pi,n>0} F_{\mu,\alpha,n}(v_{\mu+n\alpha})$ and $E_\mu(v)$ is the vector with  $E_{\mu,\alpha,n}(v)\in M_{\mu+n\alpha}$ as the entry at  the place $(\alpha,n)$. 
The final axiom is the following.
\begin{enumerate}
\item[(X3)] For any $\mu\in X$ we have $M_\mu=\ker E_\mu\oplus\im F_\mu$.
\end{enumerate}
We call the elements in $\ker E_\mu$ the {\em primitive vectors} and  the elements in $\im F_\mu$ the {\em coprimitive} vectors in $M_\mu$.

Now we can define the category $\CX_I$ for any closed subset $I$.
\begin{definition} The objects in  $\CX_{I}$ are  $I$-graded $K$-vector spaces $M=\bigoplus_{\mu\in I}M_\mu$ endowed with $K$-linear endomomorphisms $E_{\mu,\alpha,n}\colon M_\mu\to M_{\mu+n\alpha}$ and $F_{\mu,\alpha,n}\colon M_{\mu+n\alpha}\to M_\mu$ for all  $\mu\in I$, $\alpha\in\Pi$ and $n>0$, for which the conditions (X1), (X2) and (X3) are satisfied. A morphism $f\colon M\to N$  in $\CX_{I}$  is a  homogeneous $K$-linear map from $M$ to $N$ with graded components $f_\mu\colon M_\mu\to N_\mu$, that commutes with all $E$- and $F$-homomorphisms, i.e.~the diagrams

\centerline{
\xymatrix{
M_{\mu+n\alpha}\ar[r]^{f_{\mu+n\alpha}}\ar[d]_{F^M_{\alpha,n}}&N_{\mu+n\alpha}\ar[d]^{F^N_{\alpha,n}}\\
M_\mu\ar[r]^{f_\mu}&N_\mu
}
\quad\quad
\xymatrix{
M_{\mu+n\alpha}\ar[r]^{f_{\mu+n\alpha}}&N_{\mu+n\alpha}\\
M_{\mu}\ar[u]^{E^M_{\alpha,n}}\ar[r]^{f_{\mu}}&N_{\mu}\ar[u]_{E^N_{\alpha,n}}
}
}
\noindent
commute for all $\mu\in I$, $\alpha\in\Pi$ and $n>0$.
\end{definition}
In the case $I=X$ we write $\CX$ instead of $\CX_X$.

\begin{remark}\label{rem-mordelta}
If $M$ and $N$ are objects in $\CX_I$ and $f=\{f_\mu\colon M_\mu\to N_\mu\}_{\mu\in I}$ is a collection of homomorphisms, we denote  by $f_{\delta\mu}\colon M_{\delta\mu}\to N_{\delta\mu}$ the diagonal matrix with entries $f_{\mu+n\alpha}$. Then $f$ is a morphism in $\CX_I$ if and only if for all $\mu\in I$ the diagrams

\centerline{
\xymatrix{
M_{\delta\mu}\ar[d]_{F^M_\mu}\ar[r]^{f_{\delta\mu}}&N_{\delta\mu}\ar[d]^{F^N_\mu}\\
M_\mu\ar[r]^{f_\mu}& N_\mu
}
\quad\quad
\xymatrix{
M_{\delta\mu}\ar[r]^{f_{\delta\mu}}&N_{\delta\mu}\\
M_\mu\ar[u]^{E^M_\mu}\ar[r]^{f_\mu}& N_\mu\ar[u]_{E^N_\mu}
}
}
\noindent commute.
\end{remark}

Note that for two objects $M$ and $N$ in $\CX$ one can define their direct sum $M\oplus N$ in the obvious way. But due to axiom (X1)   the category $\CX$ is not closed under taking arbitrary direct sums. We can consider an arbitrary direct sum of objects in $\CX$ as an object in $\CX$ as long as each weight space is finite dimensional and the set of weights is quasi-bounded.
%However, certain infinite direct sums can be defined, for example the object $\bigoplus_{n\ge 0} S(\lambda_n)$, where $\lambda_0>\lambda_1>\lambda_2>\dots$.

%The remainder of this section is devoted to a proof of the above theorem. 
\subsection{The endomorphism $G_{\delta\mu}$} The main idea for the following is that for any $\mu\in X$ and any object $M$ of $\CX$, the composition $E_\mu\circ F_\mu\colon M_{\delta\mu}\to M_{\delta\mu}$ is already determined by the operators $F_{\nu,\alpha,n}$, $E_{\nu,\alpha,n}$ with $\nu>\mu$. This is due to the fact that the matrix entries of $E_\mu\circ F_\mu$ (with respect to the decomposition $M_{\delta\mu}=\bigoplus_{\alpha\in\Pi, n>0} M_{\mu+n\alpha}$) are the homomorphisms $E_{\mu,\beta,m}\circ F_{\mu,\alpha,n}\colon M_{\mu+n\alpha}\to M_\mu\to M_{\mu+m\beta}$. Using the commutation relations (X2) we can rewrite this homomorphism in terms of $E$- and $F$-operators that only operate on spaces with weight $\nu>\mu$. Using axiom (X3) this determines $F_\mu$ as well as the restriction of $E_\mu$ to $\im F_\mu$, so we already determined everything on a direct summand of $M_\mu$, i.e.~everything up to the primitive vectors in $M_\mu$.

Let $I\subset X$ be a closed subset and let $\mu\in X$ be such that $\mu+n\alpha\in I$ for all $\alpha\in\Pi$ and $n>0$. Let $M$ be an object in $\CX_I$. Then 
the graded vector space $M_{\delta\mu}=\bigoplus_{\alpha\in\Pi, n>0}M_{\mu+n\alpha}$ is defined even if $\mu\not\in I$. We define the endomorphism 
$$
G_{\delta\mu}\colon M_{\delta\mu}\to M_{\delta\mu}
$$ 
with the following matrix coefficients. For $\alpha,\beta\in\Pi$, $m,n>0$, the matrix coefficient 
$(G_{\delta\mu})_{\mu+m\alpha,\mu+n\beta}\colon M_{\mu+n\beta}\to M_{\mu+m\alpha}$ is given by the right hand side of axiom (X2), i.e.
$$
(G_{\delta\mu})_{\mu+m\alpha,\mu+n\beta}
:=\begin{cases}
F_{\beta,n}E_{\alpha,m},&\text{ if $\alpha\ne\beta$},\\
\sum_{r}c_{\mu,\alpha,m,n,r}F_{\alpha,n-r}E_{\alpha,m-r},&\text{ if $\alpha=\beta$}.
\end{cases}
$$
We  let $\widehat F_\mu\colon M_{\delta\mu}\to \im G_{\delta\mu}$ be the corestriction of $G_{\delta\mu}$ onto its image, and we let $\widehat E_\mu\colon \im G_{\delta\mu}\to M_{\delta\mu}$ be the inclusion. 

\begin{lemma} \label{lemma-imFG} If $\mu\in I$, then 
$G_{\delta\mu}=F_\mu\circ E_\mu\colon M_{\delta\mu}\to M_{\delta\mu}$. In this case
there exists a unique $K$-linear isomorphism $\gamma_\mu\colon \im F^M_\mu\to \im G_{\delta\mu}$ such that the diagrams

\centerline{
\xymatrix{
&M_{\delta\mu}\ar[dl]_{F_\mu}\ar[dr]^{\widehat F_\mu}&\\
\im F_\mu\ar[rr]^{\gamma_\mu}&& \im G_{\delta\mu}
}
\quad
\xymatrix{
&M_{\delta\mu}&\\
\im F_\mu\ar[ur]^{E_\mu}\ar[rr]^{\gamma_\mu}&& \im G_{\delta\mu}\ar[ul]_{\widehat E_\mu}
}
}
\noindent
commute.
\end{lemma}
\begin{proof} The very definition of $G_{\delta\mu}$ shows that  the first claim is just a reformulation of the commutation relations in axiom (X2). So let us prove the second statement. By axiom (X3) we have $M_\mu=\im F_\mu\oplus \ker E_\mu$. Hence $E_\mu$ is injective when restricted to $\im F_\mu$. Hence the statement  follows from $G_{\delta\mu}=E_\mu\circ F_\mu$.
\end{proof}

\subsection{Restriction functors}
Let $I^\prime\subset I\subset X$ be closed subsets of $X$. Let $M$ be an object in $\CX_{I}$. We now associate an object $M^\prime$ in $\CX_{I^\prime}$ to $M$ in the following (obvious) way. We let $M^\prime:=\bigoplus_{\mu\in I^\prime} M_\mu$ be the restriction of the grading to the set $I^\prime$ and we forget all homomorphisms $E_{\mu,\alpha,n}$ and $F_{\mu,\alpha,n}$ with $\mu\not\in I^\prime$. This yields a {\em restriction functor} 
$$
\Res=\Res_I^{I^{\prime}}\colon\CX_{I}\to\CX_{I^\prime}.
$$

\begin{lemma} \label{lem-extmor} Let $I^\prime\subset I$ be closed subsets of $X$ and  let $M,N$ be objects in $\CX_I$. 
\begin{enumerate}
\item The functorial map
$$
\Hom_{\CX_I}(M,N)\to\Hom_{\CX_{I^\prime}}(\Res\, M,\Res\, N)
$$
is surjective.
\item  Suppose that for all $\mu\in I\setminus I^\prime$ the homomorphism $F_\mu^M\colon M_{\delta\mu}\to M_\mu$ is surjective. Then the functorial map in (1) is a bijection. 
\end{enumerate}
\end{lemma}

\begin{proof} As the sets of weights of $M$ and $N$  are quasi-bounded from above, it is sufficient to consider the case $I=I^\prime\cup\{\mu\}$ for some $\mu\not\in I^\prime$. Let us write $M^\prime$ and $N^\prime$ instead of $\Res\,  M$ and $\Res\,  N$.  Let $f^\prime\colon M^\prime\to N^\prime$ be a morphism in $\CX_{I^\prime}$. Then $f^\prime$ induces a linear map $f^\prime_{\delta\mu}\colon M_{\delta\mu}\to N_{\delta\mu}$ (note that we can identify $M_{\delta\mu}$ and $N_{\delta\mu}$ with $M^\prime_{\delta\mu}$ and $N^\prime_{\delta\mu}$).  We need to find a $K$-linear homomorphism $f_\mu\colon M_\mu\to N_\mu$ such that the diagrams 

\centerline{
\xymatrix{
M_{\delta\mu}\ar[d]_{F^M_\mu}\ar[r]^{f^\prime_{\delta\mu}}&N_{\delta\mu}\ar[d]^{F^N_\mu}\\
M_{\mu}\ar[r]^{f_\mu} & N_{\mu}
}
\quad
\xymatrix{
M_{\delta\mu}\ar[r]^{f^\prime_{\delta\mu}}&N_{\delta\mu}\\
M_{\mu}\ar[u]^{E^M_\mu}\ar[r]^{f_\mu} & N_{\mu}\ar[u]_{E^N_\mu}
}
}
\noindent
commute. Note that if the homomorphism $F^M_\mu\colon M_{\delta\mu}\to M_\mu$ is surjective, then there can exist at most one such linear map $f_\mu$. Hence (2) follows from (1).

As $f^\prime$ is a morphism in $\CX_{I^\prime}$, the map $f^\prime_{\delta\mu}$ commutes with the endomorphisms $G_{\delta\mu}^M$ and $G_{\delta\mu}^N$. Hence there is hence a unique homomorphism $\widehat f_\mu\colon \im G_{\delta\mu}^M\to \im G_{\delta\mu}^N$ such that the diagrams 

\centerline{
\xymatrix{
M_{\delta\mu}\ar[d]_{\widehat F^M_\mu}\ar[r]^{f^\prime_{\delta\mu}}&N_{\delta\mu}\ar[d]^{\widehat F^N_\mu}\\
\im G^M_{\delta\mu}\ar[r]^{\widehat f_\mu} & \im G^N_{\delta\mu}
}
\quad
\xymatrix{
M_{\delta\mu}\ar[r]^{f^\prime_{\delta\mu}}&N_{\delta\mu}\\
\im G^M_{\delta\mu}\ar[u]^{\widehat E^M_\mu}\ar[r]^{\widehat f_\mu} & \im G^N_{\delta\mu}\ar[u]_{\widehat E^N_\mu}
}
}
\noindent
commute. By Lemma \ref{lemma-imFG} we can identify $\im G_{\delta\mu}^?$ with $\im F_\mu^?$ in such a way that the diagrams

\centerline{
\xymatrix{
M_{\delta\mu}\ar[d]_{F^M_\mu}\ar[r]^{f^\prime_{\delta\mu}}&N_{\delta\mu}\ar[d]^{F^N_\mu}\\
\im F^M_{\mu}\ar[r]^{\widehat f_\mu} & \im F^N_{\mu}
}
\quad
\xymatrix{
M_{\delta\mu}\ar[r]^{f^\prime_{\delta\mu}}&N_{\delta\mu}\\
\im F^M_{\mu}\ar[u]^{E^M_\mu}\ar[r]^{\widehat f_\mu} & \im F^N_\mu\ar[u]_{E^N_\mu}
}

}
\noindent
commute. As $M_\mu=\im F^M_\mu\oplus \ker E^M_\mu$ we obtain an extension $f_\mu$ of $f^\prime$ by extending $\widehat f_\mu$ by zero on the direct summand $\ker E^M_\mu$.  This proves (1). 
\end{proof}

\subsection{Extension functors}\label{subsec-ext}  
Let $I^\prime\subset I\subset X$ be closed subsets of $X$.
\begin{proposition}\label{prop-ext} There exists a functor $\mathsf{E}=\mathsf{E}_{I^\prime}^I\colon \CX_{I^\prime}\to\CX_I$ that is left adjoint to $\Res=\Res_I^{I^\prime}\colon \CX_I\to\CX_{I^\prime}$. It has the following properties.
\begin{enumerate}
\item The adjunction morphism $\id\to\Res\circ\mathsf{E}$ is an isomorphism of functors.
\item For all $M\in \CX_{I^\prime}$ and $\mu\in I\setminus I^\prime$ the homomorphism $F_\mu\colon (\mathsf{E}\, M)_{\delta\mu}\to(\mathsf{E}\, M)_{\mu}$ is surjective.
\end{enumerate}
\end{proposition}
\begin{proof} First note that if $I^{\prime\prime}\subset I^\prime\subset I$ are closed subsets and if we have  functors $\Extend_{I^{\prime\prime}}^{I^\prime}$ and $\Extend_{I^\prime}^I$ that have the properties stated in above, then their composition $\Extend_{I^{\prime\prime}}^I$ is a functor satisfying the above as well. So we can use induction on the size of the set $I^\prime\subset I$. However, since $X$ is a set that is not bounded from above, it is not clear where to start the induction. Suppose we want to extend an object $M^\prime$ of $\CX_{I^\prime}$, i.e.~we want to find an object $M$ in $\CX_I$ with $\Res_I^{I^\prime} M\cong M^\prime$.  Clearly we have to set $M_\mu=M^\prime_\mu$, $E_{\mu,\alpha,n}^{M}=E_{\mu,\alpha,n}^{M^\prime}$,  $F_{\mu,\alpha,n}^{M}=F_{\mu,\alpha,n}^{M^\prime}$ for all $\mu\in I^\prime$, $\alpha\in\Pi$, $n>0$. If $\mu\in I$ is such that there is no weight $\lambda$ of $M^\prime$ with $\mu\le\lambda$, then we can set $M_\mu=0$. Hence we can start our inductive procedure with setting $M_\mu=0$ for all such $\mu$  (and, of course $E_{\mu,\alpha,n}^M=0$ and $F_{\mu,\alpha,n}^M=0$ for all $\alpha,n$). 
If all weights in $I\setminus I^\prime$ satisfy this, then it is obvious that $M$ is an object in $\CX_I$ with the claimed properties and that the construction is functorial. 

For the remaining weights we can now proceed inductively (as the set of weights of $M^\prime$ is quasi-bounded from above). So it suffices to consider the case 
$I=I^\prime\cup\{\mu\}$ for some $\mu\not\in I^\prime$. Let $M^\prime$ be an object in $\CX_{I^\prime}$. We define $M_\nu$, $E_{\nu,\alpha,n}$, $F_{\nu,\alpha,n}$ as before for all $\nu\in I^\prime$. Then we can already define $M_{\delta\mu}$ ($=M^\prime_{\delta\mu}$) and its endomorphism $G_{\delta\mu}$.  We now define $M_\mu:=\im G_{\delta\mu}$ and $F_\mu=\widehat F_\mu\colon M_{\delta\mu}\to M_\mu$ as the corestriction of $G_{\delta\mu}$ to its image, and $E_\mu=\widehat E_\mu\colon M_\mu\to M_{\delta\mu}$ as the inclusion. We claim that the object $M$ belongs to  $\CX_I$. The axiom (X1) is clearly satisfied. By construction, we have $E_\mu\circ F_\mu=G_{\delta\mu}$. Looking at the individual matrix entries we realize that this equation encodes the commutation relations (X2). Clearly we have $\im F_\mu=M_\mu$ and $\ker E_\mu=0$. Hence (X3) holds. 
So  we have indeed defined an object $M\in \CX_I$. 
 
 Now let $f^\prime\colon M^\prime\to N^\prime$ be a morphism in $\CX_{I^\prime}$. It induces a homomorphism $f^\prime_{\delta\mu}\colon M^\prime_{\delta\mu}\to N^\prime_{\delta\mu}$ with the property that the diagram
 
 \centerline{
\xymatrix{
M^\prime_{\delta\mu}\ar[d]_{G_{\delta\mu}^{M^\prime}}\ar[r]^{f^\prime_{\delta\mu}}&N^\prime_{\delta\mu}\ar[d]^{G_{\delta\mu}^{N^\prime}}\\
M^\prime_{\delta\mu}\ar[r]^{f_{\delta\mu}} & N^\prime_{\delta\mu}
}
}
\noindent
commutes. From this we deduce that $f^\prime$ induces a homomorphism $f_\mu\colon M_\mu\to N_\mu$ on the images of the $G_{\delta\mu}$-homomorphisms. Then we deduce that the diagrams 

 \centerline{
\xymatrix{
M^\prime_{\delta\mu}\ar[d]_{F^M_{\mu}}\ar[r]^{f^\prime_{\delta\mu}}&N^\prime_{\delta\mu}\ar[d]^{F_{\mu}^{N}}\\
M_{\mu}\ar[r]^{f_{\mu}} & N_{\mu}
}\quad
\xymatrix{
M^\prime_{\delta\mu}\ar[r]^{f^\prime_{\delta\mu}}&N^\prime_{\delta\mu}\\
M_{\mu}\ar[r]^{f_{\mu}} \ar[u]^{E^M_{\mu}}& N_{\mu}\ar[u]_{E_{\mu}^{N}}
}
}
\noindent
commute.
This shows that 
 the construction  is functorial, so we arrive at a functor $\mathsf{E}=\mathsf{E}_{I^\prime}^I\colon \CX_{I^\prime}\to \CX_I$. Note that it follows from the construction that it has property (2).

From the construction it is obvious that $ M^\prime\cong \Res\,  M$ functorially. Hence we obtain an isomorphism $ \id_{\CX_{I^\prime}}\to\Res\circ \mathsf{E}$ of functors. Consider the  functorial homomorphism
$$
\Hom_{\CX_{I}}(\mathsf{E}\, M^\prime,N)\to \Hom_{\CX_{I^\prime}}(\Res\circ\Extend\, M^\prime, \Res\, N).
$$
Since $\Res\circ \mathsf{E}\, M^\prime\cong M^\prime$ and since, by construction, the homomorphism $F_\mu\colon (\mathsf{E}\,  M)_{\delta\mu}\to (\mathsf{E}\,  M)_\mu$ is surjective, it follows from Lemma \ref{lem-extmor}  that this homomorphism is a bijection, i.e.~$\mathsf{E}$ is left adjoint to $\Res$. So we proved statement (1). 
\end{proof}
We will apply  the extension functor to a ``skyscraper object at $\lambda$''  to obtain the {\em standard objects} $S(\lambda)$.
\subsection{Construction of the standard objects} 
Before we construct the standard objects in $\CX$ we need the following definitions. 
Let $M$ be an object in $\CX$. 
\begin{definition} \begin{enumerate}
\item The object $M$  is called {\em $F$-cyclic} if there exists a vector $v$ in $M$ such that $M$ is  the smallest subspace of $M$ that contains $v$ and is stable under  $F_{\alpha,n}$ for all $\alpha\in\Pi$ and $n>0$. 
\item For $\lambda\in X$ we define
$$
M^{prim}_\lambda:=\{m\in M_\lambda\mid E_{\alpha,n}(m)=0\text{ for all $\alpha\in\Pi$, $n>0$}\}.
$$
This is called the set of {\em primitive vectors} of $M$ of weight $\lambda$.
\end{enumerate}
\end{definition}
\begin{remark}\label{rem-Fcycl} Suppose that $M$ is an object in $\CX$ with maximal weight $\lambda$ and that $M_\lambda$ is one-dimensional. Then $M$ is $F$-cyclic if for all $\mu\in X$, $\mu\ne \lambda$ the homomorphism $F_\mu\colon M_{\delta\mu}\to M_\mu$ is surjective.
\end{remark} 
\begin{theorem}\label{thm-constrS} \begin{enumerate}
\item For all $\lambda\in X$ there exists an up to isomorphism unique  object $S(\lambda)$ in $\CX$ with the following properties.
\begin{enumerate}
\item $S(\lambda)$ is indecomposable in $\CX$.
\item The weight space $S(\lambda)_\lambda$ is of dimension $1$  and  $S(\lambda)_\mu\ne 0$ implies $\mu\le\lambda$. 
\end{enumerate}
\item The objects $S(\lambda)$ characterized in (1) have the following additional properties.
\begin{enumerate}
\item $S(\lambda)$ is $F$-cyclic.
\item  Let $v_\lambda\in S(\lambda)_\lambda$ be a non-zero vector. Then the homomorphism 
\begin{align*}
\Hom_\CX(S(\lambda),M)&\to M_\lambda^{prim}\\
f&\mapsto f(v_\lambda)
\end{align*}
is well-defined and an isomorphism of vector spaces.
\item We have $S(\lambda)_\lambda^{prim}=S(\lambda)_\lambda$ and $S(\lambda)_\mu^{prim}=0$ for all $\mu\ne\lambda$. 
\item For $\lambda\ne\mu$ we have $\Hom_\CX(S(\lambda),S(\mu))=0$ and $\End_\CX(S(\lambda))=K\cdot \id$. 
\item  Let $M$ be an object in $\CX$. Then there exists an index set $J$ and weights $\lambda_j\in X$ for $j\in J$ such that $M\cong\bigoplus_{j\in J}S(\lambda_j)$. The multiset $\{\lambda_i\}$ is uniquely determined by $M$. 
\end{enumerate}
\end{enumerate}
\end{theorem}

\begin{proof} Let us fix $\lambda\in X$ and set $I_\lambda=X\setminus\{<\lambda\}:=\{\mu\in X\mid\mu\not<\lambda\}$. This is a closed subset of $X$ that contains $\lambda$ as a minimal element. Then we define an object $S^\prime(\lambda)\in\CX_{I_\lambda}$ as the skyscraper at $\lambda$, i.e.~we set $S^\prime(\lambda)_\lambda=Kv_\lambda$ and $S^\prime(\lambda)_\mu=0$ if $\mu\in I_\lambda$, $\mu\ne\lambda$. All $E$- and $F$-homomorphisms have to be zero, of course. This indeed defines an object in $\CX_{I_\lambda}$ and the homomorphism
\begin{align*}
\Hom_{\CX_{I_\lambda}}(S^\prime(\lambda),M)&\to M^{prim}_\lambda,\\
f&\mapsto f_\lambda(v_\lambda)
\end{align*}
is well-defined and a bijection for all objects $M$ of $\CX_{I_\lambda}$.

Now we use the extension functor  from Proposition \ref{prop-ext} and  define
$$
S(\lambda):=\mathsf{E}_{I_\lambda}^X S^\prime(\lambda).
$$
By construction, this is an object in $\CX$. Let $\mu\in I_\lambda$. Then $S(\lambda)_\mu=(\Res_X^{I_\lambda} S(\lambda))_\mu=S^\prime(\lambda)_\mu$ and this vector space vanishes, if $\mu\ne\lambda$, and is of dimension $1$, if $\mu=\lambda$. Hence $S(\lambda)_\mu\ne 0$ implies either $\mu=\lambda$ or $\mu\not\in I_\lambda$, i.e.~$\mu<\lambda$. 
So we have constructed, for all $\lambda\in X$, a specific  object $S(\lambda)$ that satisfies the property (1b). We now show that these objects also satisfy all properties in (2). Then property (1)(a) and the uniqueness statement in (1) follow  from (2)(e). 
 
Property (2)(a) follows  from Remark \ref{rem-Fcycl} and the property of the extension functor that is stated in Proposition \ref{prop-ext}, (2).  Now let us prove (2)(b). From the fact that $E_{\alpha,n}(v_\lambda)=0$ for all $\alpha$, $n>0$, we deduce that the map in (2)(b) is well-defined. Moreover,
$$
\Hom_\CX(S(\lambda),M)=\Hom_\CX(\mathsf{E}_{I_\lambda}^XS^\prime(\lambda),M)=\Hom_{\CX_{I_\lambda}}(S^\prime(\lambda), \Res_X^{I_\lambda}M).
$$
As $M_\lambda^{prim}=(\Res_X^{I_\lambda}M)_\lambda^{prim}$ we deduce the statement in (2)(b), as we already observed
$\Hom_{\CX_{I_\lambda}}(S^\prime(\lambda), N)=N_\lambda^{prim}$ for all objects $N$ of $\CX_{I_\lambda}$. 

Let us show (2)(c). Let $\mu\in X$. If $\mu\in I_\lambda$, then  $S(\lambda)_\mu^{prim}=S^\prime(\lambda)_\mu^{prim}$ and this space vanishes if $\mu\ne\lambda$ and equals $S^{\prime}(\lambda)_\lambda$ if $\lambda=\mu$. If $\mu\not\in I_\lambda$, then (X3) implies that  the homomorphism $E_\mu\colon S(\lambda)_\mu\to S(\lambda)_{\delta\mu}$ is injective, as $F_\mu\colon S(\lambda)_{\delta\mu}\to S(\lambda)_\mu$ is surjective by Proposition \ref{prop-ext}. Hence $S(\lambda)_\mu^{prim}=0$,  hence (2)(c). Property (2)(d) is an easy consequence of (2)(b) and (2)(c).

Finally we prove (2)(e).  By (X1) the object  $M$ has a maximal weight $\lambda\in X$. Then there exist  $K$-linear homomorphisms $f^\prime\colon K v_\lambda \to M_\lambda$ and $g^\prime\colon M_\lambda\to K v_\lambda$ such that $g^\prime\circ f^\prime=\id_{K v_\lambda}$. The maximality of $\lambda$ implies that we can view these homomorphisms as morphisms $f^\prime\colon \Res_X^{I_\lambda} S(\lambda)\to \Res^{I_\lambda}_X M$ and $g^\prime\colon  \Res^{I_\lambda}_X M\to \Res_X^{I_\lambda} S(\lambda)$ with the property that $g^\prime\circ f^\prime=\id_{\Res_X^{I_\lambda} S(\lambda)}$. 
Lemma \ref{lem-extmor} implies that there exist morphisms $f\colon S(\lambda)\to M$ and $g\colon M\to S(\lambda)$ that extend $f^\prime$ and $g^\prime$. In particular, $g\circ f$ is a non-zero endomorphism of $S(\lambda)$. By (2)(d), $g\circ f$ is an automorphism of $S(\lambda)$. This means that we can write $M=S(\lambda)\oplus M^\prime$ for some object $M^\prime$ in $\CX$ with the property that $\dim M_{\lambda}=\dim M^\prime_\lambda+1$. Repeating the above we arrive at a decomposition $M=M^\prime\oplus \bigoplus_{i=1}^{\dim M_\lambda}S(\lambda)$ such that $M^\prime_\lambda=0$. Downwards induction on the set of primitive weights of $M$ now finishes the proof of the existence of a direct sum decomposition as in statement (2)(e). Note that if $M=\bigoplus_{j\in J}S(\lambda_j)$, then the multiplicity of $S(\lambda)$ in this decomposition equals $\dim M_\lambda^{prim}$ by $(2)(c)$,  hence it is uniquely determined by $M$.
\end{proof}

\section{Contravariant forms}
In this section we study contravariant forms on the objects in the category $\CX$. Note that in the definition of $\CX$, the roles of the $E$- and the $F$-operators are not symmetric. The existence of a non-degenerate contravariant form, however, reveals that there is some symmetry after all. But we will see  that (non-degenerate) contravariant forms only exist if the choice of coefficients is symmetric in $m$ and $n$.

\begin{definition} \label{def-sym} We say that the choice of constants $c$ is  {\em symmetric} if
for all $\mu\in X$, $\alpha\in\Pi$, $m,n>0$ and $r\in\DZ$ we have
$$
c_{\mu,\alpha,m,n,r}=c_{\mu,\alpha,n,m,r}.
$$
\end{definition}
Our main example $c_{\mu,\alpha,m,n,r}=\qchoose{\lgl\mu,\alpha^\vee\rgl+m+n}{r}$ is symmetric. 
\subsection{The definition of a contravariant form}
  Let $I$ be a closed set and $M$  an object in $\CX_{I}$. We do not need to assume yet that $c$ is symmetric. 
\begin{definition}\label{def-cvform}  A {\em contravariant form} on $M$ is a bilinear form $b\colon M\times M\to K$ with the following properties.
\begin{enumerate}
\item $b$ is symmetric.
\item The weight space decomposition is orthogonal with respect to $b$, i.e.~if $\lambda,\mu\in I$ and $\lambda\ne\mu$, then    $b(v,w)=0$ for all $v\in M_\lambda$ and $w\in M_\mu$.
\item The $E$-operators are adjoint to the $F$-operators with respect to $b$, i.e.~for $\mu\in I$,  $\alpha\in\Pi$, $n>0$, $v\in M_\mu$ and $w\in M_{\mu+n\alpha}$ we have
$$
b(E_{\alpha,n}(v),w)=b(v,F_{\alpha,n}(w)).
$$
\end{enumerate}
\end{definition}
For any bilinear form $b$ on $M$ and $\mu\in I$ we write $b_\mu$ for the restriction of $b$ to $M_\mu\times M_\mu$, and for all $\mu\in X$ such that $\mu+n\alpha\in I$ for all $\alpha\in\Pi$, $n>0$ we write $b_{\delta\mu}$ for the restriction of $b$ to $M_{\delta\mu}\times M_{\delta\mu}$. 

\begin{lemma} \label{lem-muadj} Suppose that $b$ is a bilinear form on $M$ that satisfies properties (1) and (2) of Definition \ref{def-cvform}.  Then    $b$ is a contravariant form if and only if for all $\mu\in I$, $v\in M_{\mu}$ and $w\in M_{\delta\mu}$ we have
$$
b_{\delta\mu}(E_\mu(v),w)=
b_\mu(v,F_\mu(w)).
$$
\end{lemma}
\begin{proof} Since $M_{\delta\mu}=\bigoplus_{\alpha\in\Pi\atop n>0} M_{\mu+n\alpha}$, for the condition stated in the lemma it suffices to check the identity $b_{\delta\mu}(E_\mu(v),w)=b_\mu(v,F_\mu(w))$ for all $\mu\in X$, $v\in M_{\mu}$, $\alpha\in\Pi$, $n>0$ and $w\in M_{\mu+n\alpha}$. But if $w\in M_{\mu+n\alpha}$, then $b_{\delta\mu}(E_\mu(v),w)=b(E_{\alpha,n}(v),w)$ by the orthogonality of the weight space decomposition, and $F_{\mu}(w)=F_{\alpha,n}(w)$. Then $b_{\delta\mu}(E_\mu(v),w)=b_\mu(v,F_\mu(w))$ is the same as  $b(E_{\alpha,n}(v),w)=b(v,F_{\alpha,n}(w))$ and  the claim follows. 
\end{proof}
 
 \subsection{Self-adjointness of $G_{\delta\mu}$}
 Let $M$ be an object in $\CX_I$ and suppose that  $\mu\in X$ is such that $\mu+n\alpha\in I$ for all $\alpha\in\Pi$, $n>0$. Then $M_{\delta\mu}$ and its endomorphism $G_{\delta\mu}$ are defined. 
  
 \begin{lemma} \label{lem-Gdsa} Assume that the choice of constants is symmetric.  Suppose that $b$ is a contravariant form on $M$. Then $G_{\delta\mu}$ is a self-adjoint endomorphism on $M_{\delta\mu}$ with respect to $b_{\delta\mu}$.
 \end{lemma}
 \begin{remark} In the case that $\mu\in I$ we have $G_{\delta\mu}=E_\mu\circ F_\mu$ and the statement of the lemma above follows directly from Lemma \ref{lem-muadj}. In the case $\mu\not\in I$ we have to work a little harder.  
 \end{remark}
 \begin{proof} We need to show that $b_{\delta\mu}(G_{\delta\mu}(v),w)=b_{\delta\mu}(v,G_{\delta\mu}(w))$ for all $v,w\in M_{\delta\mu}$. We can assume that $v\in M_{\mu+m\alpha}$ and $w\in M_{\mu+n\beta}$ for some $\alpha,\beta\in\Pi$, $m,n>0$. First suppose that $\alpha\ne\beta$. As the weight spaces of $M_{\delta\mu}$ are orthogonal with respect to $b_{\delta\mu}$ we have
\begin{align*}
b(G_{\delta\mu}(v),w)&=b(G_{\delta\mu}(v)_{\mu+n\beta},w)\\
&=b(F_{\alpha,m}E_{\beta,n}(v),w)\\
&=b(v,F_{\beta,n}E_{\alpha,m}(w))\\
&=b(v,G_{\delta\mu}(w)).
\end{align*}
In the case $\alpha=\beta$ we calculate
\begin{align*}
b(G_{\delta\mu}(v),w)&=b(G_{\delta\mu}(v)_{\mu+n\alpha},w)\\
&=b(\sum_rc_{\mu,\alpha,m,n,r}F_{\alpha,m-r}E_{\alpha,n-r}(v),w)\\
&=b(v,\sum_rc_{\mu,\alpha,m,n,r}F_{\alpha,n-r}E_{\alpha,m-r}(w))\\
&=b(v,\sum_rc_{\mu,\alpha,n,m,r}F_{\alpha,n-r}E_{\alpha,m-r}(w))\\
&=b(v,G_{\delta\mu}(w)).
\end{align*}
Note that in the fourth equation above we used the fact that the function $c$ is symmetric. 
 \end{proof}
 
\subsection{Extension of contravariant forms}

Let $I^\prime\subset I\subset X$ be closed subsets of $X$. Let $M^\prime$ be an object in $\CX_{I^\prime}$ and let $M=\mathsf{E}_{I^\prime}^I M^\prime$ be its $I$-extension.

\begin{proposition} \label{prop-extcf} Suppose that the choice of constants is symmetric. Suppose that $b^\prime$ is a contravariant form on $M^\prime$. Then there exists a unique contravariant form $b$ on $M$ such that $b|_{M^\prime\times M^\prime}=b^\prime$. Moreover, if $b^\prime$ is non-degenerate, then so is $b$. 
\end{proposition}
\begin{proof} Again we can assume that $I=I^\prime\cup\{\mu\}$ with $\mu\not\in I^\prime$. Let us denote by $b^\prime_{\delta\mu}$ the restriction of $b^\prime$ to $M^\prime_{\delta\mu}\subset M^\prime$. We define a new contravariant form $\widehat b_\mu$ on $M^\prime_{\delta\mu}\times M^\prime_{\delta\mu}$ by twisting $b^\prime_{\delta\mu}$ with $G_{\delta\mu}$, i.e.~we set
$$
\widehat b_\mu(x,y):=b_{\delta\mu}(x,G_{\delta\mu}(y))=b_{\delta\mu}(G_{\delta\mu}(x),y),
$$
where for the second equation we used Lemma \ref{lem-Gdsa}.
As $b_{\delta\mu}$ is symmetric, this is a symmetric $K$-bilinear form on $M^\prime_{\delta\mu}$. Now recall that we defined the extension $M$ of $M^\prime$ by setting $M_\mu:=\im G_{\delta\mu}$  and then identified the maps $E_\mu$ and $F_\mu$ with the canonical inclusion $\im G_{\delta\mu}\subset M_{\delta\mu}$ and the canonical homomorphism $M_{\delta\mu}\to \im G_{\delta\mu}$ onto  the image. By definition, the kernel of $G_{\delta\mu}$ is contained in the radical of $\widehat b_\mu$, hence $\widehat b_\mu$ induces a symmetric bilinear form $b_\mu$ on $M_{\mu}$. It has the property that
$$
b_\mu(F_\mu(v),F_\mu(w))=\widehat b_\mu(v,w)=b_{\delta\mu}(G_{\delta\mu}(v),w)
$$
for all $v,w\in M_{\delta\mu}$.
We extend $b^\prime$ orthogonally by $b_\mu$ and obtain a symmetric bilinear form $b$ on $M$. If $b^\prime$ was non-degenerate, then so is $b_{\delta\mu}$, and $\widehat b_\mu$ has radical $G_{\delta\mu}$, hence the induced form $b_\mu$ on $M_{\delta\mu}/\ker G_{\delta\mu}$ is non-degenerate as well. 

We now prove that $b$ is contravariant.  As $b^\prime$ is contravariant, we only need to show that 
$$
b_\mu(v,F_\mu(w))=b_{\delta\mu}(E_\mu(v),w)
$$
for all $v\in M_\mu$ and $w\in M_{\delta\mu}$. We can write $v=F_\mu(v^\prime)$ for some $v^\prime\in M_{\delta\mu}$. Then
\begin{align*}
b_\mu(v,F_\mu(w))&=b_\mu(F_\mu(v^\prime),F_\mu(w))\\
&=b_{\delta\mu}(G_{\delta\mu}(v^\prime),w)\\
&=b_{\delta\mu}(E_\mu\circ F_\mu(v^\prime),w)\\
&=b_{\delta\mu}(E_\mu(v),w).
\end{align*}
Hence $b$ is contravariant.
\end{proof}

\begin{proposition} Suppose that the choice of constants is symmetric. Let $M$ be an object in $\CX$. Then there exists a non-degenerate contravariant form on $M$.
\end{proposition}
\begin{proof} It is sufficient to prove the claim in the case that $M=S(\lambda)$ for some $\lambda\in X$. In this case consider the closed subset $I_\lambda$ as in the proof of  Theorem \ref{thm-constrS}. Then $S^\prime=\Res_X^{I_\lambda}S(\lambda)$ is a skyscraper at $\lambda$, and $S^\prime_\lambda$ is one dimensional. Choose any non-degenerate  $K$-bilinear form $b^\prime$ on the $K$-vector space $S^{\prime}_\lambda$. This can then be considered as a non-degenerate contravariant form on $S^\prime$. Proposition \ref{prop-extcf} shows that there exists a non-degenerate contravariant form $b$ on $\mathsf{E}_{I_\lambda}^X S^\prime=S(\lambda)$. 
\end{proof}
This is all we can say  for the category $\CX$ for an (almost) arbitrary choice of constants. In the remainder of this article we fix a special choice for the coefficient function $c$. 

\section{Quantum binomial coefficients}\label{sec-GbC}
In the remainder of this article we assume that the choice of coefficients function $c$ is given by  quantum binomials. Then we deduce several properties of the objects $S(\lambda)$ from arithmetic properties of these binomials. 

\subsection{Quantum integers}\label{subsec-qint} Let $v$ be an indeterminate and set $\SZ:=\DZ[v,v^{-1}]$. For $n\in\DZ$  set 
$$
[n]:=\frac{v^{n}-v^{-n}}{v-v^{-1}}=\begin{cases}
0,&\text{ if $n=0$},\\
v^{n-1}+v^{n-3}+\dots+v^{-n+1},&\text{ if $n> 0$}, \\
-v^{-n-1}-v^{-n-3}-\dots-v^{n+1},&\text{ if $n<0$}.
\end{cases}
$$
Note that $[n]=-[-n]$ for all $n\in\DZ$.
For $a,b\in\DZ$ the quantum binomial coefficient $\qchoose ab$ is defined as   
$$
\qchoose ab=\begin{cases}
\frac{[a][a-1]\cdots[a-b+1]}{[1][2]\cdots[b]},&\text{ if $b>0$},\\
1,&\text{ if $b=0$},\\
 0,&\text{ if $b<0$}.
\end{cases}
$$
This is an element in $\SZ$ for all $a,b\in\DZ$. Note that $a\ge 0$ and $\qchoose ab\ne 0$ imply  $0\le b\le a$. Under the ring involution $\ol{\cdot}\colon\SZ\to\SZ$, $v\mapsto v^{-1}$, the quantum numbers $[n]$ and the quantum binomial coefficients $\qchoose ab$ are invariant.
Under the ring homomorphism $\SZ\to\DZ$,  $v\mapsto 1$, $[n]$ is sent to $n$  and $\qchoose{n}{r}$ to $n\choose r$. %Analogously, under the ring homomorphism $\SZ\to\DZ$,  $v\mapsto -1$, $[n]$ is sent to $n$ for all odd $n\in\DZ$ and to $-n$ for all even $n$.  

An alternative definition of the quantum numbers and quantum binomials is the following. Let $w$ be another variable and let $\SZ^\prime=\DZ[w,w^{-1}]$. Then define, for $n\in\DZ$,
$$
[n]^\prime:=\frac{w^n-1}{w-1}=\begin{cases}
0,&\text{ if $n=0$},\\
1+w+\dots+w^{n-1},&\text{ if $n>0$},\\

-w^{-n}-w^{-n+1}-\dots-w^{-1},&\text{ if $n<0$}
\end{cases} 
$$
 and, for $a,b\in\DZ$,
$$
\qchoose ab^\prime= \begin{cases}
\frac{[a]^\prime[a-1]^\prime\dots[a-b+1]^\prime}{[1]^\prime[2]^\prime\cdots[b]^\prime},&\text{ if $b>0$},\\
1,& \text{ if $b=0$},\\
0,& \text{ if $b<0$}.
\end{cases}
$$
This are elements in $\SZ^\prime$.

\subsection{Binomial identities}
We start with proving several formulas for binomial coefficients.

\begin{proposition}\label{prop-binomident} For $a,b,x,y,n\in\DZ$ the following holds.
\begin{enumerate}
\item (the transformation formula): If we identify $w$ with $v^2$, then 
$$
[a]^\prime=v^{a-1}[a]\text{ and }\qchoose{a}{b}^\prime=v^{b(a-b)}\qchoose ab.
$$
\item If $b>0$, then 
$$
\qchoose ab=\qchoose a{a-b}.
$$ 
\item (the inversion formula):
$$
\qchoose ab=(-1)^b\qchoose{b-a-1}{b}.
$$
\item (the Pascal identity):
$$
\qchoose ab=v^{b}\qchoose{a-1}b+v^{b-a}\qchoose{a-1}{b-1}.
$$
\item (the Chu-Vandermonde convolution formula):  
$$
\qchoose{a+b}{n}=\sum_{r+s=n} v^{as-br}\qchoose ar\qchoose bs.
$$ 

\item (the Pfaff-Saalsch\"utz identity):
$$
\qchoose{x+a}{a}\qchoose{y+b}{b}=\sum_{k} \qchoose{x+y+k}{k}\qchoose{x+a-b}{a-k}\qchoose{y+b-a}{b-k}.
$$
 \item For all  $m\ge 2$ we have 
 $$
 \sum_{r} (-1)^rv^{r(2-m)}\qchoose{x-r}{1}\qchoose mr=0.
 $$
\end{enumerate}
\end{proposition}
\begin{proof} Formula (2) and the inversion formula (3) follow directly from the definition of the binomial coefficients, formula (1) and the Pascal identity (4) require  simple calculations. 
Let us prove (5).
The {\em   $w$-Chu-Vandermonde convolution formula}  is (see, for example, the solution to Exercise 100 in  Chapter I in \cite{Stan})
$$
\qchoose{a+b}{n}^\prime=\sum_{r+s=n} w^{(a-r)s}\qchoose ar^\prime \qchoose bs^\prime.
 $$
 Setting $w=v^2$ and using the transformation formula (1) this reads
\begin{align*}
v^{n(a+b-n)}\qchoose{a+b}{n}&=\sum_{r+s=n} v^{2(a-r)s+r(a-r)+s(b-s)}\qchoose ar \qchoose bs\\
&=\sum_{r+s=n} v^{n(a+b-n)-rb+as}\qchoose ar \qchoose bs,
 \end{align*}
 where we used
 \begin{align*}
 n(a+b-n)-rb+as&=(r+s)(a+b-r-s)-rb+as\\
 &=ra+rb-r^2-rs+as+bs-rs-s^2-rb+as\\
 &=2as-2rs+ra-r^2+sb-s^2\\
 &= 2(a-r)s+r(a-r)+s(b-s),
  \end{align*}
 and the identity (5) follows by dividing by $v^{n(a+b-n)}$.

Also the Pfaff-Saalsch\"utz identity (6) follows from its $w$-binomial counterpart. The latter  was originally proven by Jackson, but now  there are several proofs in the literature. The proof in \cite{Z} is particularly interesting, as it involves a counting argument.  The $w$-Pfaff-Saalsch\"utz identity  has several equivalent formulations, and the following version  can be found in  \cite{G}:
$$
\qchoose{x+a}{a}^\prime\qchoose{y+b}{b}^\prime=\sum_{k} w^{(a-k)(b-k)}\qchoose{x+y+k}{k}^\prime\qchoose{x+a-b}{a-k}^\prime\qchoose{y+b-a}{b-k}^\prime.
$$
Using the transformation formula in (1)  above we get
$$
v^{ax+by}\qchoose{x+a}{a}\qchoose{y+b}{b} =\sum_{k}v^{c_k}\qchoose{x+y+k}{k}\qchoose{x+a-b}{a-k}\qchoose{y+b-a}{b-k}
$$
with 
\begin{align*}
c_k&=2(a-k)(b-k)+k(x+y)+(a-k)(x-b+k)+(b-k)(y-a+k)\\
&=k(x+y)+(a-k)(x+b-k)+(b-k)(y-a+k)\\
&=k(x+y)+(a-k)x+(b-k)y\\
&=ax+by.
\end{align*}
Hence the $w$-Saalsch\"utz identity can be rewritten as 
$$
\qchoose{x+a}{a}\qchoose{y+b}{b}=\sum_{k} \qchoose{x+y+k}{k}\qchoose{x+a-b}{a-k}\qchoose{y+b-a}{b-k}
$$
which is identity (6). 

%Note that it suffices to show the formula for the $+$-sign. The $-$-sign version follows by applying the automorphism $\ol{\cdot}\colon\SZ\to\SZ$, $v\mapsto v^{-1}$ that leaves the quantum numbers and hence the quantum binomial numbers invariant. 
We show that (7) follows from the Saalsch\"utz identity (6) by induction on $m\ge 2$. For $a=b=y=1$ the Saalsch\"utz identity reads
$$
\qchoose{x+1}{1}\qchoose{2}{1}=\qchoose{x+1}{0}\qchoose{x}{1}\qchoose{1}{1}+ \qchoose{x+2}{1}\qchoose{x}{0}\qchoose{1}{0}
$$
or
$$
\qchoose 21\qchoose{x+1}{1} =\qchoose x1+\qchoose{x+2}1.
$$
This is the case $m=2$ in formula (6) with $x$ replaced by $x+2$. So we can take this as the starting point of an inductive argument. 

Suppose that $m\ge 3$ and that $\sum_{r} (-1)^rv^{r(2-(m-1))}\qchoose{a-r}{1}\qchoose {m-1}r=0$ is proven. If we replace  $\qchoose mr$ by $v^r\qchoose {m-1}r+v^{r-m}\qchoose{m-1}{r-1}$ (Pascal's identity (4)) in the expression $\sum_{r} (-1)^rv^{r(2-m)}\qchoose{a-r}{1}\qchoose mr$ we obtain
\begin{align*}
&\sum_{r} (-1)^rv^{r(2-m)}\qchoose{a-r}{1}\left( v^r\qchoose{m-1}r +v^{r-m}\qchoose{m-1}{r-1}\right)\\
&=\sum_{r} (-1)^rv^{r(2-m)+r}\qchoose{a-r}{1}\qchoose{m-1}r+ \\
&\quad +\sum_{r} (-1)^rv^{r(2-m)+r-m}\qchoose{a-r}{1}\qchoose{m-1}{r-1} \\
&=\sum_{r} (-1)^rv^{r(2-(m-1))}\qchoose{a-r}{1}\qchoose{m-1}r + \\
&\quad -v^{3-2m}\sum_{r} (-1)^{r-1}v^{(r-1)(2-(m-1))}\qchoose{(a-1)-(r-1)}{1}\qchoose{m-1}{r-1} \\
&= 0-v^{3-2m}0=0
\end{align*}
using the induction hypothesis and the following identity:
\begin{align*}
3-2m+(r-1)(2-(m-1))&=3-2m+(r-1)(3-m)\\
&=-2m+3r-rm+m\\
&=r(2-m)+r-m.
\end{align*}
\end{proof}

\subsection{The  quantum characteristic}

Now let $K$ be a field (of arbitrary characteristic) and $q\in K^\times$ an invertible element. From now on we consider the quantum integers $[n]$ and the quantum binomials $\qchoose ab$ as elements in the field $K$ via the ring homomorphism $\DZ[v,v^{-1}]\to K$ that sends $v$ to $q$. The formulas in Proposition \ref{prop-binomident} then hold if we replace $v$ by $q$. 

\begin{definition} We define the {\em quantum characteristic} $\ell\ge 0$ of the pair $(K,q)$ as follows.
\begin{enumerate}
\item We set $\ell=0$ if  $[n]\ne 0$ in $K$ for all $n\ne 0$. 
\item Otherwise $\ell$ is the smallest positive integer with $[\ell]=0$ in $K$.
\end{enumerate}
\end{definition}
Note that $[-n]=-[n]$ and that $[1]=1$ in all cases, so either $\ell=0$ or $\ell\ge 2$. If $\ell=2$, then $[2]=q^{-1}+q=0$, so $q^2=-1$, so $q$ is a primitive $4$-th root of unity if $\ch\, K\ne 2$, or $q=1$ if $\ch\, K=2$.

\begin{lemma} \label{lemma-rou} Suppose that $\ell>0$. Then $q$ is a  $2\ell$-th root of unity in $K$. Moreover, the following holds.
\begin{enumerate}
\item If $q=1$ or $q=-1$, then $K$ is a field of positive characteristic, and $\ell=\ch\, K$.
\item Suppose that $q\ne \pm 1$ and that the order of $q$  is odd. Then $\ell$ equals the order of $q$, i.e.~$q$ is a primitive $\ell$-th root of unity. Moreover, $[n]=0$ in $K$ if and only if $n\in\ell\DZ$.
\end{enumerate}
\end{lemma}

\begin{proof} Note that $[n]=\frac{v^n-v^{-n}}{v-v^{-1}}=v^{1-n}\frac{v^{2n}-1}{v^2-1}$. So if the image of  $[n]$  vanishes in $K$, then $q^{2n}=1$, so either $n=0$ or  $q$ is a $2n$-th root of unity. If $q=\pm 1$, then $[n]=\pm n$, so  $[n]$ vanishes if and only if $n$ is a multiple of the (ordinary) characteristic of $K$. So we have proven (1).  If $q$ is odd, then  $q^{2n}=1$ is equivalent to   $q^n=1$, so (2).
\end{proof}

\subsection{Binomial identities in positive quantum characteristics}
We list a few additional identities  that hold in positive quantum characteristics.  
\begin{proposition}[The $q$-Lucas Theorem] \label{prop-qLuc} Suppose that $\ell>0$ and that the order of $q$ is  odd if $q\ne\pm 1$. Let $a,b\in\DZ$ and write $a=a_0+\ell a_1$, $b=b_0+\ell b_1$ with $0\le a_0,b_0<\ell$. Then
$$
\qchoose ab=\qchoose{a_0}{b_0}{a_1\choose b_1}.
$$
\end{proposition}
Note that $a \choose b$ here stands for the ordinary binomial coefficient, i.e.~the $q=1$ version of the quantum binomial coefficient! 
\begin{proof} The $w$-version of the $q$-Lucas theorem reads
$$
\qchoose ab^\prime=\qchoose{a_0}{b_0}^\prime{a_1\choose b_1}^\prime
$$
(cf.~\cite{Des}). Using the transformation formula in Proposition \ref{prop-binomident} this gives us
$$
v^{b(a-b)}\qchoose ab=v^{b_0(a_0-b_0)}\qchoose{a_0}{b_0}{a_1\choose b_1}.
$$
As $v^\ell=1$ by Lemma \ref{lemma-rou}, $v^{b(a-b)}=v^{b_0(a_0-b_0)}$ and we obtain the claimed identity.
\end{proof}

The following lists some simple conclusions.
\begin{lemma}\label{lem-binom3} Suppose that $\ell>0$ and the order of $q$ is odd if $q\ne\pm1$. Let $a,b,n\in\DZ$. Then the following holds.
\begin{enumerate}
\item $$
\qchoose {\ell a}b=\begin{cases}
0,& \text{if $b\not\in \ell\DZ$},\\
{a\choose b/\ell},&\text{ if $b\in\ell\DZ$}.
\end{cases}
$$
\item 
$$\qchoose{a+\ell b}{n}=\sum_{n=r+\ell s} \qchoose{a}{r}{ b\choose s}.
$$ 
\end{enumerate}  
\end{lemma}

\begin{proof} The statement (1) follows directly from Lucas' theorem (Proposition \ref{prop-qLuc}). In order to prove (2) we use the Chu-Vandermonde convolution formula  in Proposition \ref{prop-binomident}:
$$
\qchoose{a+\ell b}{n}=\sum_{r+t=n}q^{at-\ell br}\qchoose{a}{r}\qchoose{\ell b}{t}.
$$
By (1), $\qchoose{\ell b}{t}=0$ unless $t$ is of the form $t=\ell s$ for some $s\in\DZ$, in which case $\qchoose{\ell b}{t}={b\choose s}$. If $t=\ell s$, then $q^{at-\ell br}=1$ as $q$ is an $\ell$-th root of unity by Lemma \ref{lemma-rou}. Hence 
$$
\qchoose{a+\ell b}{n}=\sum_{r+\ell s=n}\qchoose{a}{r}{b\choose s}.
$$ 
\end{proof}

\section{Characteristic independent relations}
For the rest of the paper we consider the category $\CX$ as defined in Section \ref{sec-gradspaces} with $c$ given by certain binomial coefficients.  For simplicity, we  assume that the root system $R$ is {\em simply laced},  i.e.~$\lgl\alpha,\beta^\vee\rgl\in\{0,-1\}$ if $\alpha\ne\beta$. In the non-simply laced case, the choice of $c$ has to be slightly altered, and some of the binomial identities that we would need for the following results are not available in the literature. 

\subsection{The choice of coefficients}
We set
$$
c_{\mu,\alpha,m,n,r}:=\qchoose{\lgl\mu,\alpha^\vee\rgl+m+n}{r}
$$
for all $\mu\in X$, $\alpha\in\Pi$, $m,n>0$ and $r\in\DZ$. Note that this choice is symmetric in the sense of Definition \ref{def-sym}. It is now also convenient to slightly rewrite the axiom (X2). The version (X2) was necessary as we started out with the definition of $\CX_I$ for a closed subset $I$ of $X$. Now we only need the global case $I=X$. The new axiom reads
\begin{enumerate}
\item[(X2)${}^\prime$] For all  $\mu\in X$, $\alpha,\beta\in\Pi$, $m,n>0$ and $v\in M_{\mu}$ we have
$$
E_{\alpha,m}F_{\beta,n}(v)=\begin{cases}
F_{\beta,n}E_{\alpha,m}(v),&\text{ if $\alpha\ne\beta$},\\
\sum_r
\qchoose{\lgl\mu,\alpha^\vee\rgl+m-n}{r}
F_{\alpha,n-r}E_{\alpha,m-r}(v),&\text{ if $\alpha=\beta$}.
\end{cases}
$$
\end{enumerate} 
We replaced the weight $\mu+n\beta$ with the weight $\mu$ in the formulation of the axiom. Hence the coefficient $c_{\mu,\alpha,m,n,r}$ is replaced with $c_{\mu-n\alpha,m,n,r}$ and hence $\qchoose{\lgl\mu,\alpha^\vee\rgl+m+n}{r}$  with $\qchoose{\lgl\mu-n\alpha,\alpha^\vee\rgl+m+n}{r}=\qchoose{\lgl\mu,\alpha^\vee\rgl+m-n}{r}$.

We now apply the binomial identities that we obtained in the previous chapter to understand the arithmetics of the $E$- and $F$- operators on the objects in $\CX$. 

\subsection{Dominant weights}
 We start with a relatively simple property of the objects $S(\lambda)$ in the case that  $\lambda\in X$ is {\em dominant}, i.e.~satisfies $\lgl\lambda,\alpha^\vee\rgl\ge 0$ for all $\alpha\in\Pi$. 

\begin{lemma}\label{lemma-domweight} Let $\lambda\in X$ be a dominant weight and $v\in S(\lambda)_\lambda$. Then $F_{\alpha,n}(v)=0$ for all $n>\lgl\lambda,\alpha^\vee\rgl$.
\end{lemma}
\begin{proof} Suppose that $n>\lgl\lambda,\alpha^\vee\rgl$. We show that $E_{\gamma,s}F_{\alpha,n}(v)=0$ for all $\gamma\in\Pi$ and $s>0$. Then $F_{\alpha,n}(v)=0$ by axiom (X3).  If $\gamma\ne\alpha$, then $E_{\gamma,s}F_{\alpha,n}(v)=F_{\alpha,n}E_{\gamma,s}(v)=0$ as $\lambda+s\gamma$ is not a   weight of $S(\lambda)$. If $\gamma=\alpha$, then
\begin{align*}
E_{\alpha,s}F_{\alpha,n}(v)&=\sum_r\qchoose{\lgl\lambda,\alpha^\vee\rgl+s-n}{r}F_{\alpha,n-r}E_{\alpha,s-r}(v)\\
&=\qchoose{\lgl\lambda,\alpha^\vee\rgl+s-n}{s}F_{\alpha,n-s}(v)
\end{align*}
as $E_{\alpha,t}(v)=0$ for all $t>0$. For all $n$ such that $\lgl\lambda,\alpha^\vee\rgl<n\le\lgl\lambda,\alpha^\vee\rgl+s$ we have $0\le \lgl\lambda,\alpha^\vee\rgl+s-n<s$, hence $\qchoose{\lgl\lambda,\alpha^\vee\rgl+s-n}{s}=0$. For all $n$ such that $n>\lgl\lambda,\alpha^\vee\rgl+s$, hence $n-s>\lgl\lambda,\alpha^\vee\rgl$, we can use the formula that we obtain by  induction and deduce $F_{\alpha,n-s}(v)=0$. 
\end{proof}

\subsection{Divided powers}
In order to simplify notation, we use the following convention. If we write down a relation between operators (for example, $F_{\alpha,m}F_{\alpha,n}=\qchoose{m+n}{m} F_{\alpha,m+n}$) we  mean, more precisely, that the relation holds if we apply it to any element of any object  $M$ of  $\CX$ (for example,  $F_{\alpha,m}F_{\alpha,n}(v)=\qchoose{m+n}{m} F_{\alpha,m+n}(v)$ for all $M$ in $\CX$ and all $v\in M$).

\begin{lemma} \label{lemma-divpow} Let $\alpha\in\Pi$ and $m,n\ge 0$. Then
\begin{align*}
E_{\alpha,m}E_{\alpha,n}&=\qchoose{m+n}{m} E_{\alpha,m+n},\\
F_{\alpha,m}F_{\alpha,n}&=\qchoose{m+n}{m} F_{\alpha,m+n}.
\end{align*}
In particular:
\begin{enumerate}
\item The operators $E_{\alpha,m}$ and $E_{\alpha,n}$ commute for all $m,n\ge 0$, and the operators $F_{\alpha,m}$ and $F_{\alpha,n}$ commute for all $m,n\ge 0$.
\item ${[1][2]\dots[n]}E_{\alpha,n}=E_{\alpha,1}^n$  and ${[1][2]\dots[n]}F_{\alpha,n}=F_{\alpha,1}^n$ for all $n\ge 0$.
\end{enumerate}
\end{lemma}

\begin{proof} The claims (1) and (2) are easy consequences of the two displayed identities. These certainly hold in the case $m+n=0$. So we assume $m+n>0$. Fix an object $M\in \CX$, some $\mu\in X$ and $w\in M_\mu$. We need to show that $E_{\alpha,m}E_{\alpha,n}(w)=\qchoose{m+n}{m} E_{\alpha,m+n}(w)$ and $F_{\alpha,m}F_{\alpha,n}(w)=\qchoose{m+n}{m} F_{\alpha,m+n}(w)$. Using the fact that there exists a non-degenerate contravariant form on $M$ it is sufficient to prove the identity for the $F$-operators. Using $m+n>0$ we can use axiom (X3) and deduce that it is sufficient to show that 
\begin{equation}\label{eqn-one}
 E_{\gamma,s}\left(F_{\alpha,m}F_{\alpha,n}(w)\right)=E_{\gamma,s}\left(\qchoose{m+n}{m} F_{\alpha,m+n}(w)\right)
\end{equation} 
for all $\gamma\in\Pi$ and $s>0$. 

First suppose that $\gamma\ne\alpha$. Then we can commute $E_{\gamma,s}$ past all $F_\alpha$-homomorphisms to the right and realize that we need to prove 
$$
F_{\alpha,m}F_{\alpha,n}E_{\gamma,s}(w)=\qchoose{m+n}{m} F_{\alpha,m+n}E_{\gamma,s}(w).
$$ 
This certainly holds if $E_{\gamma,s}(w)=0$. Using axiom (X1) we can now argue by downwards induction on the weight $\mu$ and deduce the claimed identity.

Now we consider the case $\gamma=\alpha$.  Again we commute the homomorphism $E_{\alpha,s}$ to the right using the commutation relations (X2). The left hand side of equation (\ref{eqn-one}) becomes
$$
\sum_{u,v} c_u d_{u,v} F_{\alpha,m-u}F_{\alpha,n-v}(w_{s-(u+v)})
$$
where  for notational simplicity we set $w_{t}=E_{\alpha,t}(w)$ for all $t$ and 
\begin{align*}
c_u&=\qchoose{\lgl\mu-n\alpha,\alpha^\vee\rgl+s-m}{u}\\
&=\qchoose{\lgl\mu,\alpha^\vee\rgl+s-m-2n}{u}\\
&=\qchoose{\chi-n}{u},\\
d_{u,v}&=\qchoose{\lgl\mu,\alpha^\vee\rgl+s-u-n}{v}\\
&=\qchoose{\chi+m-u}{v}.\\
\end{align*}
with $\chi=\lgl\mu,\alpha^\vee\rgl+s-(m+n)$.

We now use induction on the weight of $w$ as before and also on $m+n$. Then we can replace $F_{\alpha,m-u}F_{\alpha,n-v}(w_{s-(u+v)})$ with $\qchoose{m+n-(u+v)}{m-u}F_{\alpha,m+n-(u+v)}(w_{s-(u+v)})$, as for $u+v=0$ the weight of $w_{s}$ is strictly larger than $\mu$, and for $u+v>0$ we have $m+n-(u+v)<m+n$. So the left hand side of equation (\ref{eqn-one}) is
\begin{align*}
LHS&=\sum_{u,v} c_{u}d_{u,v}\qchoose{m+n-(u+v)}{m-u}F_{\alpha,m+n-(u+v)}(w_{s-(u+v)})\\
&=\sum_{r,u} \qchoose{\chi-n}{u}\qchoose{\chi+m-u}{r-u}\qchoose{m+n-r}{m-u}F_{\alpha,m+n-r}(w_{s-r})
\end{align*}
where we replaced the variable $v$ with $r=u+v$.

The right  hand side of equation (\ref{eqn-one}) becomes after applying the commutation relations
\begin{align*}
RHS&=\qchoose{m+n}{m} \sum_{r} \qchoose{\lgl\mu,\alpha^\vee\rgl+s-(m+n)}{r}F_{\alpha,m+n-r}E_{\alpha,s-r}(w)\\
&=\qchoose{m+n}{m} \sum_{r} \qchoose{\chi}{r}F_{\alpha,m+n-r}(w_{s-r}).
\end{align*}

Now we fix $r$ and show that  the coefficients of $F_{\alpha,m+n-r}(w_{s-r})$ in the expressions LHS and RHS coincide. Hence we need to show that 
$$
\sum_{u} \qchoose{\chi-n}{u}\qchoose{\chi+m-u}{r-u}\qchoose{m+n-r}{m-u}=
\qchoose{m+n}{m}\qchoose{\chi}{r}
$$
or, if we replace $u$ with $k=r-u$ and use $\qchoose{m+n-r}{m-u}=\qchoose{m+n-r}{n+u-r}=\qchoose{m+n-r}{n-k}$,
$$
\sum_{k} \qchoose{\chi-n}{r-k}\qchoose{\chi+m-r+k}{k}\qchoose{m+n-r}{n-k}=
\qchoose{m+n}{n}\qchoose{\chi}{r}.
$$
This is the Saalsch\"utz identity from Proposition \ref{prop-binomident}, i.e.~the identity
$$
\sum_{k} \qchoose{x+y+k}{k}\qchoose{x+a-b}{a-k}\qchoose{y+b-a}{b-k}=\qchoose{x+a}{a}\qchoose{y+b}{b}
$$
with $x=\chi-r$, $a=r$, $y=m$ and $b=n$.
\end{proof}

\subsection{The Serre-Lusztig relations}
The case $m=2$ in the following proposition yields the original {\em Serre relations}. The {\em higher Serre relations}, i.e.~the cases with $m>2$, were proven by Lusztig in the case of quantum groups (\cite[Chapter 1.4]{LQG}).  Lusztig proved an even more general identity that we do not need for the following.
 \begin{proposition}\label{prop-SLRel} Let $\alpha,\beta\in\Pi$, $\alpha\ne\beta$, and $m,n\ge0$.
\begin{enumerate}
\item If $\lgl\alpha,\beta^\vee\rgl=0$, then $F_{\alpha,m}F_{\beta,n}=F_{\beta,n}F_{\alpha,m}$ and $E_{\alpha,m}E_{\beta,n}=E_{\beta,n}E_{\alpha,m}$.
\item If $\lgl\alpha,\beta^\vee\rgl=-1$ and $m\ge 2$, then 
\begin{align*}
\sum_{r} (-1)^r q^{ r(2-m)}F_{\alpha,r}F_{\beta,1}F_{\alpha,m-r} &=0,\\
\sum_{r} (-1)^r q^{r(2-m)}E_{\alpha,r}E_{\beta,1}E_{\alpha,m-r} &=0.
\end{align*}
\end{enumerate}
\end{proposition}
\begin{proof}  Again fix an object $M$ of $\CX$, a weight $\mu$ and an element $w\in M_\mu$ and prove both identities by showing that both sides yield the same vector when applied  to $w$.  We only prove the versions for the $F$-operators. The $E$-operator version then follows from the existence of a non-degenerate form on $M$ and the fact that the $E$-operators are adjoint to the $F$-operators with respect to this form.

We start with identity (1). Again we prove the claim by showing that $E_{\gamma,s}F_{\alpha,m}F_{\beta,n}(w)=E_{\gamma,s}F_{\beta,n}F_{\alpha,m}(w)$ for all $\gamma\in\Pi$ and $s>0$. If $\gamma\not\in\{\alpha,\beta\}$, then we can commute the $E$-homomorphism to the far right, and  downwards induction on $\mu$ yields the claim. So suppose that $\gamma=\alpha$. Then commuting the $E$-homomorphism to the right on the left hand side of the equation yields
$$
\sum_{u}\qchoose{\lgl\mu-n\beta,\alpha^\vee\rgl+s-m}{u}F_{\alpha,m-u}F_{\beta,n}E_{\gamma,s-u}(w).
$$
Doing the same thing to the right hand side yields
$$
\sum_{v}\qchoose{\lgl\mu,\alpha^\vee\rgl+s-m}{v}F_{\beta,n}F_{\alpha,m-v}E_{\gamma,s-v}(w).
$$

As $\lgl\beta,\alpha^\vee\rgl=0$, the binomial coefficient of the $u$-summand in the first and the $v$-summand in the second expression coincide if we identify $u=v$.  Downward induction on the weight $\mu$ now finishes the argument.
The case $\gamma=\beta$ is treated in a symmetric fashion. This proves statement (1). 

Now we prove statement (2). We need to show that %It is sufficient to prove the $+$-sign version as the $-$-sign version follows by applying the involution $v\mapsto v^{-1}$. We need to show that 
\begin{equation}\label{eqn-two}
E_{\gamma,s}\left(\sum_{r} (-1)^r q^{r(2-m)}F_{\alpha,r}F_{\beta,1}F_{\alpha,m-r} (w)\right)=0
\end{equation}
for all $\gamma\in \Pi$ and $s>0$. If $\gamma\ne\alpha$, $\gamma\ne\beta$, then we commute $E$ past all $F$-maps and then use induction on the weight $\mu$ of $w$. 

Now suppose that $\gamma=\alpha$. Using the commutation relations we obtain
$$E_{\alpha,s}F_{\alpha,r}F_{\beta,1}F_{\alpha,m-r} (w) = \sum_{u,v} c_{u,r} d_{u,v,r} F_{\alpha,r-u}F_{\beta,1}F_{\alpha,m-r-v} (w_{s-(u+v)}) 
$$
with $w_{t}=E_{\alpha,t}(w)$ and (note that $\lgl\alpha,\beta^\vee\rgl=-1$!)
\begin{align*}
c_{u,r}&=\qchoose{\lgl\mu-(m-r)\alpha-\beta,\alpha^\vee\rgl-r+s}{u}\\
&= \qchoose{\lgl\mu,\alpha^\vee\rgl-2m+r+s+1}{u}\\
&=\qchoose{\chi-m+r+1}{u},\\
d_{u,v,r}&=\qchoose{\lgl\mu,\alpha^\vee\rgl+s-u-m+r}{v}\\
&=\qchoose{\chi-u+r}{v}
\end{align*}
where we abbreviate $\chi:=\lgl\mu,\alpha^\vee\rgl-m+s$.  Hence the left hand side of  equation (\ref{eqn-two}) becomes 
$$\sum_{r,u,v}  (-1)^r q^{r(2-m)} \qchoose{\chi-m+r+1}{u}\qchoose{\chi-u+r}{v}F_{\alpha,r-u}F_{\beta,1}F_{\alpha,m-r-v}(w_{s-u-v}). 
$$
Now we replace $(u,v)$ by $(c,d)$ with $c:=u+v$ and $d:=r-u$. Then $m-r-v=m-c-d$ and the above expression reads
\begin{equation}\label{eqn-three}
\sum_{r,c,d}  (-1)^r q^{r(2-m)} \qchoose{\chi-m+r+1}{r-d}\qchoose{\chi+d}{c+d-r}F_{\alpha,d}F_{\beta,1}F_{\alpha,m-c-d}(w_{s-c})
\end{equation}
Now let us fix $c$ and $d$. The coefficient in front of the element  $F_{\alpha,d}F_{\beta,1}F_{\alpha,m-c-d}(w_{s-c})$ is  
$$
\sum_{r}  (-1)^r q^{r(2-m)} \qchoose{\chi-m+r+1}{r-d}\qchoose{\chi+d}{c+d-r}
$$
or, with $t=r-d$,
$$
\sum_{t}  (-1)^{t+d} q^{(t+d)(2-m)} \qchoose{\chi-m+t+d+1}{t}\qchoose{\chi+d}{c-t}.
$$
Using the inversion formula from  Proposition  \ref{prop-binomident} this  equals
$$
\sum_{t}  (-1)^{d} q^{(t+d)(2-m)} \qchoose{-\chi+m-d-2}{t}\qchoose{\chi+d}{c-t}.
$$
or
\begin{equation}\label{eqn-four}
(-1)^dv^{d(2-m)}\sum_{t}  q^{t(2-m)} \qchoose{-\chi-d+m-2}{t}\qchoose{\chi+d}{c-t}.
\end{equation}
We leave this expression for a moment. 
A special case of the Chu-Vandermonde convolution yields
$$
\sum_{t}  q^{(c-t)(-\chi-d+m-2)-t(\chi+d)} \qchoose{-\chi-d+m-2}{t}\qchoose{\chi+d}{c-t}=\qchoose{m-2}{c}.
$$
 We calculate $(c-t)(-\chi-d+m-2)-t(\chi+d)=-c(\chi+d+2-m)+t(2-m)$ and hence
$$
\sum_{t}  q^{t(2-m)} \qchoose{-\chi-d+m-2}{t}\qchoose{\chi+d}{c-t}=q^{c(\chi+d+2-m)}\qchoose{m-2}{c}.
$$
So expression (\ref{eqn-four}) now is
$$
(-1)^dq^{d(2-m)+c(\chi+d+2-m)}\qchoose{m-2}{c}.
$$
Recall that we fixed $c$ and $d$ in expression (\ref{eqn-three}) and took the summation over $r$ to obtain this expression. 
We plug this into equation (\ref{eqn-three}) and obtain
$$
\sum_{c,d}  (-1)^dq^{d(2-m)+c(\chi+d+2-m)}\qchoose{m-2}{c}F_{\alpha,d}F_{\beta,1}F_{\alpha,m-c-d}(w_{s-c}).
$$
Now we fix $c$. The summation over $d$ is the expression
$$
q^{c(\chi+2-m)} \qchoose{m-2}{c} \sum_{d} (-1)^dq^{d(2-(m-c))}F_{\alpha,d}F_{\beta,1}F_{\alpha,m-c-d}(w_{s-c}).
$$
Suppose that $c=0$. Then we can use downwards induction on the weight of $w$ and deduce that this expression vanishes (note that $s>0$ and $w_{s}=E_{\alpha,s}(w)$, which vanishes if the weight of $w$ is maximal). Suppose that  $c>0$. If  $m<2+c$, the binomial coefficient vanishes. In particular, this settles the case $m=2$. Now we can use upwards induction on $m$ and deduce that $\sum_{d} (-1)^dq^{d(2-(m-c))}F_{\alpha,d}F_{\beta,1}F_{\alpha,m-c-d}=0$. Hence the above expression vanishes for all $c$ and $m\ge 2$. This is what we wanted to show in the case  $\gamma=\alpha$.

Now suppose that $\gamma=\beta$. 
Then
\begin{align*}
E_{\beta,s}F_{\alpha,r}F_{\beta,1} F_{\alpha,m-r} (w) &=F_{\alpha,r}F_{\beta,1} F_{\alpha,m-r} (w_s) +  c_r F_{\alpha,r} F_{\alpha,m-r} (w_{s-1})\\
&=F_{\alpha,r}F_{\beta,1} F_{\alpha,m-r} (w_s) +  c_r \qchoose{m}{r}F_{\alpha,m} (w_{s-1})
\end{align*}
with 
\begin{align*}
c_r&=\qchoose{\lgl\mu-(m-r)\alpha,\beta^\vee\rgl+s-1}{1} \\
&=\qchoose{\lgl\mu,\beta^\vee\rgl+m-r+s-1}{1}=\qchoose{\chi-r}{1} 
\end{align*}
with $\chi=\lgl\mu,\beta^\vee\rgl+m+s-1$.
 Using downwards induction on the weight $\mu$ it hence suffices to show that 
 $$
\sum_{r} (-1)^r q^{r(2-m)} \qchoose{\chi-r}{1} \qchoose{m}{r} F_{\alpha,m} (w_{s-1})=0.
 $$
It is sufficient to show that 
 $$
 \sum_{r} (-1)^r q^{r(2-m)} \qchoose{\chi-r}{1} \qchoose{m}{r} =0
 $$
 But this is one of the identities listed in Proposition \ref{prop-binomident}. 
\end{proof}

\section{Further relations in positive quantum characteristics}
The relations that we obtained so far hold for all pairs $(K,q)$, i.e.~they are independent of the (quantum) characteristic $\ell$. We now add further relations in the case that $\ell>0$ and  the order of $q$ is odd if $q\ne\pm1$.  

\subsection{Decomposition of the operators}
For $0\le n<\ell$ define the operators $E_{\alpha,1}^{[n]}$ and $F_{\alpha,1}^{[n]}$  inductively by setting $E_{\alpha,1}^{[0]}=\id$, $F_{\alpha,1}^{[0]}=\id$  and for $0<n<\ell$
$$
E_{\alpha,1}^{[n]}:=\frac{1}{[n]}E_{\alpha,1}E_{\alpha,1}^{[n-1]}, \quad F_{\alpha,1}^{[n]}:=\frac{1}{[n]}F_{\alpha,1}F_{\alpha,1}^{[n-1]}.
$$
%For $0\le n$ define the operators $E_{\alpha,\ell}^{(n)}$ and $F_{\alpha,\ell}^{(n)}$  inductively by setting $E_{\alpha,\ell}^{(0)}=\id$ and  $F_{\alpha,\ell}^{(0)}=\id$   and for $n>0$
%$$
%E_{\alpha,\ell}^{(n)}:=\frac{1}{n}E_{\alpha,1}E_{\alpha,1}^{(n-1)}, \quad F_{\alpha,\ell}^{(n)}:=\frac{1}{n}F_{\alpha,1}F_{\alpha,1}^{(n-1)}.
%$$

\begin{lemma} 
\label{lemma-decomp}  
Suppose that $\ell>0$ and the order of $q$ is odd if $q\ne\pm1$.  %and that the characteristic of $K$ is $0$ if $q\ne \pm 1$. 
Let $n\ge 0$ and write $n=n_0+\ell n_1$ with $0\le n_0<\ell$ and $n_1\in\DZ$. Then
\begin{align*}
E_{\alpha,n}&=E_{\alpha,1}^{[n_0]}E_{\alpha,\ell n_1}=E_{\alpha,\ell n_1}E_{\alpha,1}^{[n_0]}\\
F_{\alpha,n}&=F_{\alpha,1}^{[n_0]}F_{\alpha,\ell n_1}=F_{\alpha,\ell n_1}F_{\alpha,1}^{[n_0]}.
\end{align*}
\end{lemma}
\begin{proof} Lemma \ref{lemma-divpow} yields  $E_{\alpha,n}=\qchoose{n}{n_0}E_{\alpha,n_0}E_{\alpha,\ell n_1}$, and $\qchoose{n}{n_0}=\qchoose{n_0}{n_0}{n_1\choose 0}=1$ by Lucas' theorem (Proposition \ref{prop-qLuc}). Hence we are left with showing that $E_{\alpha,n}=E_{\alpha,1}^{[n]}$ for $0\le n<\ell$. For this we use induction on $n$. The statement is clear for $n=0$. Using Lemma \ref{lemma-divpow} again gives us $\qchoose{n}{1}E_{\alpha,n}=E_{\alpha,1}E_{\alpha,n-1}$. Since $\qchoose{n}{1}=[n]\ne 0$ for $1\le n<\ell$, we obtain 
$$
E_{\alpha,n}=\frac{1}{[n]}E_{\alpha,1}E_{\alpha,n-1}=\frac{1}{[n]}E_{\alpha,1}E_{\alpha,1}^{[n-1]}=E_{\alpha,1}^{[n]}
$$
using the induction hypothesis. We prove the identity for the $F$-operators in the same way. 
\end{proof}

\begin{lemma}
\label{lem-Fcycl}  Suppose that $\ell>0$ and the order of $q$ is odd if $q\ne\pm1$. 
Let $\lambda\in X$.
\begin{enumerate}
\item Suppose that $v\in S(\lambda)$ is such that $E_{\alpha,1}(v)=0$ and $E_{\alpha,\ell t} (v)=0$ for all $\alpha\in\Pi$ and $t>0$. Then $v\in S(\lambda)_\lambda$. 
\item Let $S\subset S(\lambda)$ be the smallest $X$-graded subspace that contains $S(\lambda)_\lambda$ and is stable under all operators $F_{\alpha,1}$ and $F_{\alpha,\ell t}$ for all $\alpha\in\Pi$ and $t>0$. Then $S=S(\lambda)$. 
\end{enumerate}
\end{lemma}
\begin{proof} First let us show that claim (1) follows from claim (2). So let $v\in S(\lambda)$ be as in claim (1). Let $b$ be a non-degenerate contravariant form on $S(\lambda)$. Then $b(v,F_{\alpha,1}(w))=b(E_{\alpha,1}(v),w)=0$ and $b(v,F_{\alpha,\ell t}(w))=b(E_{\alpha,\ell t}(v),w)=0$ for all $w\in S(\lambda)$. So if claim (2) is true, then $v$ is orthogonal to $\bigoplus_{\mu<\lambda} S(\lambda)_\mu$, hence must be contained in $S(\lambda)_\lambda$. 

So let us prove claim (2). Lemma \ref{lemma-decomp} shows that $S$ is stable under all homomorphisms $F_{\alpha,n}$ with $\alpha\in\Pi$ and $n>0$. As $S(\lambda)$ is $F$-cyclic (Theorem \ref{thm-constrS}), we deduce $S=S(\lambda)$.
\end{proof}

\subsection{On primitive and coprimitive vectors in the case of a restricted highest weight} 
The next results will be used in the proof of Steinberg's tensor product theorem. It concerns the simple objects with {\em restricted} highest weight. 
\begin{definition} A weight $\lambda\in X$ is called {\em restricted} if $0\le \lgl\lambda,\alpha^\vee\rgl<\ell$ for all $\alpha\in\Pi$.
\end{definition}
In particular, if $\ell=0$, there is no restricted weight.

\begin{lemma}\label{lemma-rescycl} Suppose that $\ell>0$ and the order of $q$ is odd if $q\ne\pm1$.  Let $\lambda\in X$ be a restricted weight. Then the following holds.
\begin{enumerate}
\item Suppose that $v\in S(\lambda)$ is such that  $E_{\alpha,1}(v)=0$ for all $\alpha\in\Pi$. Then  $v\in S(\lambda)_\lambda$.
\item Let $S\subset S(\lambda)_\mu$ be the smallest $X$-graded $K$-vector space that contains $S(\lambda)_\lambda$ and is stable under all endomorphisms $F_{\alpha,1}$ with $\alpha\in\Pi$. Then $S=S(\lambda)$.
\end{enumerate}
\end{lemma}

\begin{proof} As in the proof of Lemma \ref{lem-Fcycl} one can show that claim (2) implies claim  (1). So let us prove claim  (2). From Lemma \ref{lem-Fcycl} we deduce that it suffices to show that for any $\mu\in X$ and  $w\in S(\lambda)_\mu$, the element $F_{\alpha,\ell t}(w)$ is contained in $S$ for all $\alpha\in\Pi$ and $t>0$. We prove this by downwards induction on $\mu$. In the maximal case we have  $\mu=\lambda$ and $F_{\alpha,\ell t}(w)=0$ by Lemma \ref{lemma-domweight}  and the fact that $\lambda$ is restricted. Now suppose that $\mu\ne\lambda$ and that the claim is proven for all $\nu$ with $\mu<\nu\le\lambda$. Then we can assume that $w$ is of the form $F_{\beta,1}(w^\prime)$ for some $w^\prime\in M_{\mu+\beta}$ (by the induction hypothesis). The Serre-Lusztig relation in Proposition \ref{prop-SLRel} for $m=\ell t$ reads 
$$
\sum_{r} (-1)^r q^{r(2-\ell t)}F_{\alpha,r}F_{\beta,1}F_{\alpha,\ell t-r} =0,
$$
hence $F_{\alpha,\ell t}(w)=F_{\alpha,\ell t}F_{\beta,1}(w^\prime)$ is a linear combination of  vectors of the form $F_{\alpha,r}F_{\beta,1}F_{\alpha,\ell t-r}(w^\prime)$ with $0\le r<\ell t$. Using the induction hypothesis on $\mu$,  we obtain that $F_{\alpha,\ell t-r}(w^\prime)\in S$, hence $F_{\beta,1}F_{\alpha,\ell t-r}(w^\prime)\in S$. Now we can use induction on $t$ and deduce that $F_{\alpha,r}F_{\beta,1}F_{\alpha,\ell t-r}(w^\prime)$ for all $r$ with $0\le r<\ell t$. Hence $F_{\alpha,\ell t}(w)\in S$. \end{proof}

\subsection{The Frobenius pull-back} We still assume that the characteristic of $(K,q)$ is $\ell>0$ and that the order of $q$ is odd if $q\ne \pm1$. In this section we denote the corresponding category by $\CX_{(K,q)}$, since we will also consider the category $\CX_{(K,1)}$. We do not assume anything on the characteristic of $(K,1)$. 
 We construct a $K$-linear functor from the category $\CX_{(K,1)}$ to the category $\CX_{(K,q)}$. 
 
 So let $\lambda\in X$ and consider the simple object $S_{(K,1)}(\lambda)$ of $\CX_{(K,1)}$. We denote by $S^{\prime}$ the $X$-graded $K$-vector space with 
$$
S^\prime_\mu=
\begin{cases}
0,&\text{ if $\mu\not\in \ell X$},\\
S(\lambda)_{\frac {1}{\ell} \mu},&\text{ if $\mu\in\ell X$}.
\end{cases}
$$
For $\alpha\in\Pi$ and $n\in\DZ$ and $\mu\in X$ define
\begin{align*}
E^\prime_{\mu,\alpha,n}&:=\begin{cases}
0,&\text{ if $\mu\not\in\ell X$ or $n\not\in\ell\DZ$},\\
E_{\frac{\mu}{\ell},\alpha,\frac{n}{\ell}},&\text{ if $\mu\in\ell X$ and $n\in\ell\DZ$},
\end{cases} \\
F^\prime_{\mu,\alpha,n}&:=\begin{cases}
0,&\text{ if $\mu\not\in\ell X$ or $n\not\in\ell\DZ$},\\
F_{\frac{\mu}{\ell},\alpha,\frac{n}{\ell}},&\text{ if $\mu\in\ell X$ and $n\in\ell\DZ$}.
\end{cases} 
\end{align*}
Then $E^\prime_{\mu,\alpha,n}$ is  a homomorphism from $S^\prime_\mu$ to $S^\prime_{\mu+n\alpha}$, and $F^\prime_{\mu,\alpha,n}$ is  a homomorphism from $S^\prime_{\mu+n\alpha}$ to $S^\prime_{\mu}$.

\begin{theorem} \label{thm-Frob} Suppose that $\ell>0$ and that the order of $q$ is odd if $q\ne\pm1$. Then the $X$-graded space $S^\prime$ together with the operators  $E^\prime_{\mu,\alpha,n}$ and $F^\prime_{\mu,\alpha,n}$ defined above is an object in $\CX_{(K,q)}$. It is  isomorphic to $S_{(K,q)}(\ell\lambda)$. 
\end{theorem} 
The construction is obviously functorially, so we can consider the above as  a {\em Frobenius pull-back functor} $\Frob^\ast\colon\CX_{(K,1)}\to\CX_{(K,q)}$ with the property $\Frob^\ast(S_{(K,1)}(\lambda))\cong S_{(K,q)}(\ell\lambda)$.
\begin{proof} We show that the axioms (X1), (X2)${}^\prime$ and (X3) are satisfied. Axiom (X1) readily follows from the corresponding property of $S_{(K,1)}(\lambda)$. We now check the commutation relations. So let $\alpha,\beta\in\Pi$, $m,n>0$, $\mu\in X$ and $v\in S^\prime_\mu$. We need to check that
$$
E^\prime_{\alpha,m}F^\prime_{\beta,n}(v)=
\begin{cases}
F^\prime_{\beta,n}E^\prime_{\alpha,m}(v),&\text{ if $\alpha\ne\beta$},\\
\sum_r\qchoose{\lgl\mu,\alpha^\vee\rgl+m-n}{r} F^\prime_{\alpha,n-r}E^\prime_{\alpha,m-r}(v),&\text{ if $\alpha=\beta$}.
\end{cases}
$$
If $\mu\not\in\ell X$, then $v=0$ and both sides of the above equation vanish. Now suppose that $\mu\in\ell X$. If  $m\not\in\ell\DZ$ or $n\not\in\ell\DZ$, then the left hand side of the equation vanishes. The only terms on the right hand side that do not obviously vanish are of the form $\qchoose{\lgl\mu,\alpha^\vee\rgl+m-n}{r} F^\prime_{\alpha,n-r}E^\prime_{\alpha,m-r}(v)$, where $m-r\in\ell\DZ$ and $n-r\in\ell\DZ$. For those terms we have $m-n\in\ell\DZ$ and hence $\lgl\mu,\alpha^\vee\rgl+m-n\in\ell\DZ$.  Then $\qchoose{\lgl\mu,\alpha^\vee\rgl+m-n}{r} =0$ unless $r\in\ell\DZ$ by Lemma \ref{lem-binom3}. But if $r\in\ell\DZ$, then $m-r\in\ell\DZ$ and $n-r\in\ell\DZ$ imply $m,n\in\ell\DZ$, contrary to our assumption. 

We are left with the case  that $\mu\in\ell\DZ$ and $m,n\in\ell\DZ$. In this case the left hand side is $E^\prime_{\alpha,m}F^\prime_{\beta,n}(v)=E_{\alpha,\frac m{\ell}} F_{\beta,\frac{n}{\ell}}(v)$. In the case $\alpha\ne\beta$ the claimed identity follows from $E_{\alpha,\frac m{\ell}} F_{\beta,\frac{n}{\ell}}(v)=F_{\beta,\frac{n}{\ell}}E_{\alpha,\frac m{\ell}} (v)=F^\prime_{\beta,n}E^\prime_{\alpha,m}(v)$. Suppose that $\alpha=\beta$. As $\qchoose{\lgl\mu,\alpha^\vee\rgl+m-n}{r}=0$ unless $r=\ell s$ for some $s\in\DZ$, we obtain for $\sum_r\qchoose{\lgl\mu,\alpha^\vee\rgl+m-n}{r} F^\prime_{\alpha,n-r}E^\prime_{\alpha,m-r}(v)$:
\begin{align*}
&=\sum_s\qchoose{\lgl\mu,\alpha^\vee\rgl+m-n}{\ell s} F^\prime_{\alpha,n-\ell s}E^\prime_{\alpha,m-\ell s}(v)\\
&=\sum_s{\frac{\lgl\mu,\alpha^\vee\rgl+m-n}{\ell}\choose s} F_{\alpha,\frac{n}{\ell}-s}E_{\alpha,\frac{m}{\ell}-s}(v)\quad\text{(Lemma \ref{lem-binom3})}\\
&=\sum_s\qchoose{\frac{\lgl\mu,\alpha^\vee\rgl+m-n}{\ell}}{s}_{q=1} F_{\alpha,\frac{n}{\ell}-s}E_{\alpha,\frac{m}{\ell}-s}(v)\\
&=E_{\alpha,\frac{m}{\ell}}F_{\alpha,\frac{n}{\ell}}(v) \quad\text{(by axiom (X2)${}^\prime$ for $S_{(K,1)}(\lambda))$}\\
&=E^\prime_{\alpha,m}F^\prime_{\alpha,n}(v).
\end{align*}

Finally, let us check that axiom (X3) is also satisfied. The fact that $S_{(K,1)}(\lambda)$ is $F$-cyclic implies immediately that $S^\prime$ is $F^\prime$-cyclic. Hence $\im F^\prime_\mu=S^\prime_\mu$ for all $\mu\ne\ell\lambda$ and $\im F^\prime_{\ell\lambda}=0$. Suppose that $\mu\in X$ and $v\in S^\prime_\mu$ are such that $v\ne 0$ and $E^\prime_\mu(v)=0$. This implies that $\mu\in\ell X$ and $E_{\frac{\mu}{\ell}}(v)=0$ in $S_{(K,1)}(\lambda)$. Hence $\frac{\mu}{\ell}=\lambda$, i.e.~$\mu=\ell\lambda$. Hence $\ker E_\mu=0$ unless $\mu=\ell\lambda$, and $\ker E_{\ell\lambda}=S^\prime_{\ell\lambda}$. In any case we have $S^\prime_\mu=\im F^\prime_\mu\oplus\ker E^\prime_\mu$. Hence axiom (X3) is satisfied as well. 

So $S^\prime$ is an object in $\CX_{(K,q)}$. As it is $F$-cyclic with highest weight $\ell\lambda$, it is isomorphic to $S_{(K,q)}(\ell\lambda)$. 
\end{proof}
\begin{remark}\label{rem-FPB}
We deduce that $S_{(K,q)}(\ell\lambda)_{\mu}=0$ unless $\mu\in\ell X$, and that $E_{\alpha,n}=0$ and $F_{\alpha,n}=0$  on $S_{(K,q)}(\ell\lambda)$ for all $\alpha\in\Pi$ if $n\not\in\ell\DZ$. 
\end{remark}

\subsection{Steinberg's tensor product formula} Let $\lambda_0,\lambda_1\in X$. Now 
we consider $S^\prime:=S(\lambda_0)\otimes_K S(\ell \lambda_1)$ as an $X$-graded vector space by setting 
$S^\prime_\mu:=\bigoplus_{\mu=\nu+\ell\rho} S(\lambda_0)_\nu\otimes_K S(\ell\lambda_1)_{\ell\rho}$. (Note that we sum over all pairs $(\nu,\rho)$ here, regardless of the fact wether $\mu=\nu+\ell\rho$ is an $\ell$-adic decomposition (i.e., $\nu$ is restricted) or not.) Then we define operators $E^\prime_{\alpha,m}$ and $F^\prime_{\beta,n}$ on $S^\prime$ by setting
\begin{align*}
E^\prime_{\alpha,m}(v_0\otimes v_1)&:=\sum_{s} E_{\alpha,s}(v_0)\otimes E_{\alpha,m-s}(v_1),\\
F^\prime_{\beta,n}(v_0\otimes v_1)&:=\sum_{t} F_{\beta,t}(v_0)\otimes F_{\beta,n-t}(v_1)
\end{align*}
for $v_0\in S(\lambda_0)$ and $v_1\in S(\ell\lambda_1)$.
Recall that the operators $E_{\alpha,r}$ and $F_{\alpha,r}$ act trivially on $v_1$ unless $r\in\ell\DZ$ by Remark \ref{rem-FPB}.

\begin{theorem}\label{thm-Steinb} Suppose that $\ell>0$ and that the order of $q$ is odd if $q\ne\pm1$. Suppose that $\lambda_0$ is restricted. Then the $X$-graded space $S^\prime$ together with the operators $E_{\alpha,m}^\prime$, $F^\prime_{\beta,n}$ defined above is an object in $\CX$. It is isomorphic to $S(\lambda_0+\ell\lambda_1)$.
\end{theorem}

\begin{proof}  First we check that the axioms (X1), (X2)${}^\prime$ and (X3) are satisfied. The axiom (X1) is clear. Let us check the commutation relations of (X2)${}^\prime$. So let $\alpha,\beta\in\Pi$, $m,n>0$, $\mu_0,\mu_1\in X$, $v_0\in S(\lambda_0)_{\mu_0}$, $v_1\in S(\ell \lambda_1)_{\ell\mu_1}$. First suppose that $\alpha\ne\beta$. Then 
\begin{align*}
E^\prime_{\alpha,m}F^\prime_{\beta,n}(v_0\otimes v_1)&=\sum_{s,t} E_{\alpha,s}F_{\beta,t}(v_0)\otimes E_{\alpha,m-s}F_{\beta,n-t}(v_1)\\
&=\sum_{s,t} F_{\beta,t}E_{\alpha,s}(v_0)\otimes F_{\beta,n-t}E_{\alpha,m-s}(v_1) \\
&=F^\prime_{\beta,n}E^\prime_{\alpha,m}(v_0\otimes v_1).
\end{align*}
Suppose that $\alpha=\beta$. For convenience we now write $E_x$ and $F_y$ instead of $E_{\alpha,x}$ and $F_{\alpha,y}$. Then $E^\prime_mF^\prime_n(v_0\otimes v_1)=\sum_{s, t} E_{s}F_{t}(v_0)\otimes E_{m-s}F_{n-t}(v_1)$ and this equals
\begin{align*}
&\sum_{s,t,a,b}\qchoose{\mu_0+s-t}{a}\qchoose{\ell\mu_1+m-s-(n-t)}{b}F_{t-a}E_{s-a}(v_0)\otimes F_{n-t-b}E_{m-s-b}(v_1).
\end{align*}
We now apply the following change of variables. Set $x:=s-a$, $y:=t-a$ and $r:=a+b$. So $s=x+a$, $t=y+a$, $b=r-a$. The above expression then is 
\begin{equation*}\label{eqn-five}
\sum_{a,x,y,r}\qchoose{\mu_0+x-y}{a}\qchoose{\ell\mu_1+m-n+y-x}{r-a}F_{y}E_{x}(v_0)\otimes F_{n-y-r}E_{m-x-r}(v_1).
\end{equation*}
Now note that $F_{n-y-r}E_{m-x-r}(v_1)=0$ unless $n-y-r$ and $m-x-r$ are both divisible by $\ell$ (by Remark \ref{rem-FPB}). But then  $m-n+y-x$ is divisible by $\ell$, and this implies that  $\qchoose{\ell\mu_1+m-n+y-x}{r-a}=0$ unless $r-a$ is divisible by $\ell$. In this case, $\qchoose{\ell\mu_1+m-n+y-x}{r-a}={\frac{\ell \mu_1+m-n+y-x}{\ell}\choose\frac{r-a}{\ell}}$. Hence the expression displayed above translates into 
\begin{equation}\label{eqn-six}
\sum_{x,y,r\atop a\in r+\ell\DZ}\qchoose{\mu_0+x-y}{a}{\frac{\ell \mu_1+m-n+y-x}{\ell}\choose\frac{r-a}{\ell}}F_{y}E_{x}(v_0)\otimes F_{n-y-r}E_{m-x-r}(v_1).
\end{equation}
Now we want to fix $x,y,r$ and take the above summation over the parameter $a$. Note that  Lemma \ref{lem-binom3} yields the equation
$$
\sum_{a\in r+\ell\DZ}\qchoose{\mu_0+x-y}{a}{\frac{\ell \mu_1+m-n+y-x}{\ell}\choose\frac{r-a}{\ell}}=\qchoose{\mu_0+\ell\mu_1+m-n}{r}.
$$
Expression (\ref{eqn-six}) now simplifies to
\begin{align*}
&\sum_{x,y,r}\qchoose{\mu_0+\ell\mu_1+m-n}{r}F_{y}E_{x}(v_0)\otimes F_{n-y-r}E_{m-x-r}(v_1)\\
&=\left(\sum_{r}\qchoose{\mu_0+\ell\mu_1+m-n}{r}F^\prime_{n-r}E^\prime_{m-r}\right)(v_0\otimes v_1),
\end{align*}
which is what we wanted to show. 
Hence axiom (X2)${}^\prime$ holds.

Let us check (X3). 
We claim that $S^\prime$ is $F$-cyclic of highest weight $\lambda_0+\ell\lambda_1$. Let $S^{\prime\prime}\subset S^\prime$ be the smallest subspace that contains the (one-dimensional) subspace $S^\prime_{\lambda_0+\ell\lambda_1}=S(\lambda_0)_{\lambda_0}\otimes S(\ell\lambda_1)_{\ell\lambda_1}$ and is stable under all $F^\prime$-maps. As $F^\prime_{\alpha,1}$ acts trivially on $S(\ell\lambda_1)$ for all $\alpha$, and as $S(\lambda_0)$ is $F_1$-cyclic by Lemma \ref{lemma-rescycl} (recall that $\lambda_0$ is restricted), we deduce $S(\lambda_0)\otimes S(\ell\lambda_1)_{\ell\lambda_1}\subset S^{\prime\prime}$. The definition of the $F^\prime_{\alpha,n}$ and the fact that $S(\ell\lambda_1)$ is $F$-cyclic then shows inductively that $S(\lambda_0)\otimes S(\ell\lambda_1)_{\nu}\subset S^{\prime\prime}$ for all $\nu < \ell\lambda_1$, hence  $S(\lambda_0)\otimes S(\ell\lambda_1)\subset S^{\prime\prime}$. So $S^\prime$ is $F$-cyclic.  

Now we claim that $E^\prime_\mu\colon S^\prime_\mu\to S^\prime_{\delta\mu}$ is injective for all  $\mu\ne \lambda_0+\ell\lambda_1$. Since $E_{\alpha,1}$ acts trivially on $S(\ell\lambda_1)$ for all $\alpha\in\Pi$, we deduce
$$
\ker E_{\alpha,1}^{S^\prime}\subset \ker E_{\alpha,1}^{S(\lambda_0)}\otimes_K S(\ell \lambda_1),
$$
hence 
\begin{align*}
\ker E_{\mu}^{S^\prime}&\subset \left(\bigcap_{\alpha}\ker E_{\alpha,1}^{S(\lambda_0)}\right)\otimes_K S(\ell \lambda_1)\\
&=S(\lambda_0)_{\lambda_0}\otimes_K S(\ell \lambda_1)
\end{align*}
by Lemma \ref{lemma-rescycl}. 
Since $E_{\alpha,n}$ acts trivially on $S(\lambda_0)_{\lambda_0}$ we deduce
$$
\ker E_{\alpha,n}^{S^\prime}\subset S(\lambda_0)_{\lambda_0} \otimes_K\ker E_{\alpha,n}^{S(\ell\lambda_1)}
$$
 for all $\alpha\in\Pi$ and $n>0$. Hence
$$
\ker E_{\mu}^{S^\prime}\subset S(\lambda_0)_{\lambda_0}\otimes_K S(\ell\lambda)_{\ell\lambda_1}=S^\prime_{\lambda_0+\ell\lambda_1}.
$$
Hence $\ker E_{\mu}=0$ unless $\mu=\lambda+\ell\lambda_1$, and  $\ker E_{\lambda+\ell\lambda_1}=S^\prime_{\lambda_0+\ell\lambda_1}$ as $\lambda_0+\ell\lambda_1$ is the highest weight of $S^\prime$. 
So axiom (X3) holds as well, so $S^\prime$ is an object in $\CX$. As it is $F$-cyclic with highest weight $\lambda_0+\ell\lambda_1$, it is isomorphic to $S(\lambda_0+\ell\lambda_1)$.
\end{proof}

\section{Representations of Lie algebras and quantum groups}
In this final section we show that the category $\CX$ has a real life interpretation in the case that the coefficients $c$ are (quantum) binomials as before. Again we assume that $R$ is simply laced (for notational convenience). 

We denote by $U_\SZ$ the quantum group over $\SZ=\DZ[v,v^{-1}]$ (with divided powers) associated with the Cartan matrix $(\lgl\alpha,\beta^\vee\rgl)_{\alpha,\beta\in\Pi}$ of $R$. It is generated by the elements $e^{[n]}_\alpha,  f^{[n]}_\alpha, k_\alpha, k_\alpha^{-1}$ for $\alpha\in\Pi$ and $n>0$  and some relations that can be found  in \cite[Sections 1.1-1.3]{L90}. 
For  $\alpha\in R$,  $n>0$ also the element
$$ 
\qchoose{k_\alpha}{n}:=\prod_{s=1}^n \frac{k_\alpha v^{-s+1}-k_\alpha^{-1}v^{s-1}}{v^{s}-v^{-s}} 
$$
is contained in $U_\SZ$. We let $U_\SZ^+$, $U_\SZ^-$ and $U_\SZ^0$ be the unital subalgebras of $U_\SZ$ that are generated by the sets $\{e_\alpha^{[n]}\}$, $\{f_\alpha^{[n]}\}$ and $\{k_\alpha, k_\alpha^{-1},\qchoose{k_\alpha}{n}\}$, resp. A remarkable fact, proven by Lusztig, is that each of these subalgebras is free over $\SZ$ and admits a PBW-type basis, and that the multiplication map $U^-_\SZ\otimes_\SZ U_\SZ^0\otimes_\SZ U_\SZ^+\to U_\SZ$ is an isomorphism of $\SZ$-modules (Theorem 6.7 in \cite{L90}). 

Recall that we fixed a field $K$ and an invertible element $q\in K$. We let $U:=U_\SZ\otimes_\SZ K$ and $U^\ast:= U_\SZ^\ast\otimes_\SZ K$ for $\ast=-,0,+$.
In this article we consider $U$ only as an associative, unital algebra and forget about the Hopf algebra structure. 

By \cite[Lemma 1.1]{APW} every $\mu\in X$ yields a character
\begin{align*}
\chi_\mu\colon U^0_\SZ&\to\SZ\\
k_\alpha^{\pm1}&\mapsto v^{\pm \lgl\mu,\alpha^\vee\rgl}\\
\qchoose{k_\alpha}{r}&\mapsto \qchoose{\lgl\mu,\alpha^\vee\rgl}{r}\text{ ($\alpha\in\Pi$, $r\ge0$)}.
\end{align*}
We can extend this character to a character $\chi_\mu\colon U^0\to K$. 
A $U$-module $M$ is called a {\em weight module} if $M=\bigoplus_{\mu\in X}M_\mu$, where
$$
M_\mu:=\{m\in M\mid H.m=\chi_\mu(H)m\text{ for all $H\in U^0$}\}.
$$
Hence all the weight modules that we consider in this article  are of ``type 1'' (cf.~\cite[Section 5.1]{JanQG}). An element $\mu\in X$ is called a {\em weight of $M$} if $M_\mu\ne\{0\}$. 

The triangular decomposition of $U$ allows us to construct highest weight modules. We denote by $L(\lambda)$ the irreducible $U$-module with highest weight $\lambda$. 

Now let us consider $L(\lambda)$ as an $X$-graded $K$-vector space and 
let us denote  by $E_{\mu,\alpha,n}\colon M_\mu\to M_{\mu+n\alpha}$ and $F_{\mu,\alpha,n}\colon M_{\mu+n\alpha}\to M_\mu$ the homomorphisms given by the actions of $e_\alpha^{[n]}$ and $f_{\alpha}^{[n]}$, resp., for all $\mu\in X$, $\alpha\in\Pi$, $n>0$.

\begin{theorem} \label{thm-funS} The $X$-graded space $L(\lambda)$ together with operators $E_{\mu,\alpha,n}$ and $F_{\mu,\alpha,n}$ yields an object in $\CX$. It is isomorphic to $S(\lambda)$.
\end{theorem}
\begin{proof} We check the axoims (X1), (X2)${}^\prime$ and (X3). As the weights of $L(\lambda)$ are quasi-bounded and the weight spaces are finite dimensional, (X1) is satisfied.  If $\alpha\ne\beta$ then $e_{\alpha}$ and $f_\beta$ commute, hence $e_{\alpha}^{[m]}$ and $f_\beta^{[n]}$ commute for all $m,n>0$. In order to check the commutation relations in the case $\alpha=\beta$, set 
$$
\qchoose{k_\alpha;c}{r}=\prod_{s=1}^r\frac{k_\alpha q^{c-s+1}-k_{\alpha}^{-1}q^{-c+s-1}}{q^{s}-q^{-s}}.
$$
This element is contained in $U^0$ and acts as multiplication with 
\begin{align*}
\prod_{s=1}^r\frac{q^{\lgl\nu,\alpha^\vee\rgl+c-s+1}-q^{-\lgl\nu,\alpha^\vee\rgl-c+s-1}}{q^{s}-q^{-s}}
\end{align*}
on each vector of weight $\nu$. By \cite[Section 6.5]{L90} the following relations holds in $U_\SZ$ for all $\alpha,\beta\in\Pi$, $m,n>0$: 
 $$
 e_{\alpha}^{[m]}f_\beta^{[n]}=\sum_{r=0}^{\min(m,n)} f_\alpha^{[n-r]}\qchoose{k_\alpha;2r-m-n}{r} e_\alpha^{[m-r]}.
 $$
For  $v\in M_\mu$ we hence obtain
\begin{align*}
 e_{\alpha}^{[m]}f_\beta^{[n]}(v)&=\sum_{r=0}^{\min(m,n)} f_\alpha^{[n-r]}\prod_{s=1}^r\frac{q^{\zeta-s+1}-q^{-\zeta+s-1}}{q^{s}-q^{-s}}e_\alpha^{[m-r]}(v),
 \end{align*}
 where $\zeta=\lgl\mu+(m-r)\alpha,\alpha^\vee\rgl+2r-m-n=\lgl\mu,\alpha^\vee\rgl+m-n$. In order to prove that condition (X2) holds, it remains to show that 
 $$
 \qchoose{\zeta}{r}=\prod_{s=1}^r\frac{q^{\zeta-s+1}-q^{-\zeta+s-1}}{q^{s}-q^{-s}},
 $$
 which is (almost) immediate from the definition. Hence the axiom (X2)${}^\prime$ is satisfied.
 
 Finally, we need to check (X3). As $L(\lambda)$ is a highest weight module it is $F$-cyclic. Hence $\im F_\mu=M_\mu$ for all $\mu\ne\lambda$, and $\im F_\lambda=0$. Moreover, by the general theory in highest weight categories, $L(\lambda)$ has no primitive vectors of weight $\ne\lambda$, hence $\ker E_\mu=0$ for all $\mu\ne\lambda$, and $\ker E_\lambda=M_\lambda$. Hence (X3) is satisfied as well. We obtain the object $S(\lambda)$, as $L(\lambda)$ is $F$-cyclic with highest weight $\lambda$.
\end{proof}


\begin{thebibliography}{AMRW}
%\bibitem[AMRW]{AMRW} P.~N.~Achar, S.~Makisumi, S.~Riche, and G.~Williamson, {\em Koszul duality for Kac–-Moody groups and characters of tilting modules}, Journal of the American Mathematical Society, 32(1):261–310, 2019.
%\bibitem[A]{A} H.~H.~Andersen, {\em BGG categories in prime characteristics}, preprint, {\tt arxiv: 2106.00057}.
\bibitem[APW]{APW}  H.~H.~Andersen, P.~Polo,  K.~X.~Wen, {\em Representations of quantum algebras}, Invent.~Math.~104 (1991), no.~1, 1--59.
%\bibitem[AJS94]{AJS} Henning Haahr Andersen, Jens Carsten Jantzen, and Wolfgang Soergel, \emph{Representations of quantum  groups at a {$p$}th root of unity and of semi-simple groups in characteristic  {$p$}: independence of {$p$}}, Ast\'erisque (1994), no.~220, 321.
%\bibitem[A]{A} Henning Haahr Andersen, {\em Filtrations of Cohomology Modules}
%\bibitem[BR]{BR} R.~Bezrukavnikov, S.~Riche, {\em Hecke action on the principal block},  preprint, {\tt arxiv:2009.10587}, 2020.
%\bibitem[B]{B} {Bourbaki, Nicolas},  { \em Lie groups and {L}ie algebras. {C}hapters 7--9},  {Elements of Mathematics (Berlin)}, {Springer-Verlag, Berlin}, {2005}.
%\bibitem[C]{C} J.~Ciappara, {\em Hecke category actions via Smith-Treuman theory}, preprint, {\tt arxiv:2103.07091}, 2021.
% \bibitem[C]{C} {Chern, Shiing-shen}, {\em On a generalization of K\"ahler geometry}, in  Algebraic geometry and topology. A symposium in honor of S.~Lefschetz, pp.~103--121. Princeton University Press, Princeton, N.~J., 1957.
\bibitem[D]{Des} J.~D\'esarm\'enien, {\em Un Analogue des Congruences de Kummer pour les $q$-nombres d'Euler}, Europ.~J.~Combinatorics (1982) {\bf 3}, 19--28
%\bibitem[D]{D} S.~Donkin, {\em On tilting modules for algebraic groups}, Math.~Z.~{\bf 212} (1993), 39--60. 
%\bibitem[J]{J} Jens Carsten Jantzen, {\em Representations of Algebraic Groups}.
%\bibitem[Dem73]{Dem} Michel Demazure, \emph{Invariants sym\'etriques entiers des groupes de Weyl et torsion}, Invent.~Math.~{\bf 21} (1973), 287--302.
%\bibitem[EW14]{EW} Ben Elias and Geordie Williamson, \emph{The Hodge theory of Soergel bimodules}, Ann.~of Math.~(2) {\bf 180} (2014), no.~3, 1089--1136.
\bibitem[F1]{LefOp} P.~Fiebig, \emph{Lefschetz operators, Hodge-Riemann forms, and representations}, {International Mathematics Research Notices}, Volume 2022, Issue 4, 3031--3056.
\bibitem[F2]{TiltLoc} P.~Fiebig, \emph{Quantum tilting modules over local rings}, J.~London Math.~Soc.~(2),1--36 (2023). 
%\bibitem[Fie08b]{FieTAMS} \bysame, \emph{The combinatorics of {C}oxeter categories},   {T}rans.~{A}mer.~{M}ath.~{S}oc.~\textbf{360} (2008), 4211--4233.
%\bibitem[Fie11]{FieJAMS}\bysame, \emph{Sheaves on affine Schubert varieties, modular representations and Lusztig's conjecture}, J.~Amer.~Math.~Soc.~{\bf 24} (2011), 133--181.  
%\bibitem[Fie12]{FieCrelle}\bysame, \emph{An upper bound on the exceptional characteristics for Lusztig's character formula},   J.~Reine Angew.~Math.~{\bf 673} (2012), 1--31.
%\bibitem[F]{FracPerO} P.~Fiebig, \emph{Periodicity for subquotients of the modular category $\CO$}, to appear in Math. Z., preprint 2021, {\tt arXiv:2111.02077}. 
\bibitem[G]{G} H.~W.~Gould, {\em A new symmetrical combinatorial identity}, Journal of combinatorial theory (A) {\bf 13}. 278--286 (1972). 
%\bibitem[H]{H} Humphreys, James E., {\em  Introduction to Lie algebras and representation theory}, Graduate Texts in Mathematics, 9. Springer-Verlag, New York-Berlin, 1978. xii+171 pp.

%\bibitem[FW14]{FW14} Peter Fiebig and Geordie Williamson, \emph{Parity sheaves, moment graphs and the $p$-smooth locus of Schubert varieties},  Ann.~Inst.~Fourier {\bf 64}(2) (2014), 489--536.
\bibitem[J]{JanQG} J.~C.~Jantzen, {\em Lectures on quantum groups}, Graduate Studies in Mathematics, {\bf 6}. American Mathematical Society, Providence, RI, 1996.
%\bibitem[J2] {J} \bysame,     {\em Representations of algebraic groups}, Mathematical Surveys and Monographs { \bf 107}, {Second edition}, {American Mathematical Society, Providence, RI}, {2003}. 

%\bibitem[KT95]{KasTan} Masaki Kashiwara and Toshiyuki Tanisaki, \emph{Kazhdan--{L}usztig conjecture for   affine {L}ie algebras with negative level}, Duke Math.~J.~\textbf{77} (1995),  no.~1, 21--62. 
%\bibitem[KL79]{KLInv} David Kazhdan and George Lusztig, \emph{Representations of {C}oxeter groups and  {H}ecke algebras}, Invent.~Math.~\textbf{53} (1979), no.~2, 165--184.
%\bibitem[Lan12]{LanJoA} Martina Lanini, \emph{Kazhdan-Lusztig combinatorics in the moment graph setting}, J.~of Alg.~{\bf 370} (2012), 152--170.
%\bibitem[Lus80a]{LusProc} \bysame, \emph{Some problems in the representation theory of finite {C}hevalley  groups}, The Santa Cruz Conference on Finite Groups (Univ.~California, Santa Cruz, Calif., 1979), Proc.~Sympos.~Pure Math., vol.~37, Amer.~Math.~Soc., Providence, R.I., 1980, pp.~313--317.
%\bibitem[Lus80b]{LusAdv} George Lusztig, \emph{Hecke algebras and {J}antzen's generic decomposition  patterns}, Adv.~Math.~\textbf{37} (1980), no.~2, 121--164.
\bibitem[L1]{LQG} G.~Lusztig, {\em Introduction to quantum groups},   Modern Birkh\"{a}user Classics, Birkh\"{a}user/Springer, New York, 2010.
%\bibitem[L1]{L1} Lusztig, George, {\em Finite dimensional Hopf algebras arising from quantized universal enveloping algebras},   Journal of the AMS, vol.~3, no.~1,  (1990), p.~257--296.

\bibitem[L2]{L90} G.~Lusztig, {\em Quantum groups at roots of $1$},   Geometriae Dedicata, vol.~35,   (1990), p.~89--114.
\bibitem [L3]{L15} \bysame, {\em On the character of certain irreducible modular representations},  Represent Theory {\bf 19}, (2015), 3--8

%\bibitem[L]{L2} G.~Lusztig, {\em Quantum deformations of certain simple modules over enveloping algebras}, Adv.~Math {\bf 70} (1998), 237--249.
\bibitem[LW]{LW} Lusztig, G.; Williamson, G., {\em On the character of certain tilting modules},  Sci.~China Math.~61 (2018), no.~2, 295--298.

\bibitem[S]{Stan} R.~P.~Stanley, {\em Enumerative Combinatorics. Volume 1}, Cambridge Stud.~Adv.~ Math.~ {\bf 49}, second edition, 2012.
\bibitem[Z]{Z} D.~Zeilberger, {\em  A q-Foata Proof of the q-Saalsch\"utz Identity}, Europ.~J.~Combinatorics  (1987) {\bf 8}, 461--463.

%\bibitem[LW]{LW} G.~Lusztig, G.~Williamson, {\em Billiards and tilting characters for $\operatorname{SL}_3$},  SIGMA 14, 015, (2018).
%\bibitem[Soe90]{SoeJAMS} Wolfgang Soergel, \emph{Kategorie $\CO$, perverse Garben und Moduln {\"u}ber den Koinvarianten zur Weylgruppe}, J.~Am.~Math.~Soc.~{\bf 3} (1990), No.~2, 421--445.
%\bibitem[S]{S} Wolfgang Soergel, \emph{Roots of unity and positive characteristic},  Representations of groups (Banff, AB, 1994), CMS Conf.~Proc., vol.~16, Amer.~Math.~Soc., Providence, RI, 1995, pp.~315--338. q
%\bibitem[Soe97]{SoeRep} Wolfgang Soergel, \emph{Kazhdan--{L}usztig-{P}olynome und eine {K}ombinatorik f\"ur {K}ipp-{M}oduln}, Represent.~Theory \textbf{1} (1997), 37--68 (electronic).
%\bibitem[R]{R} C.~M.~Ringel, {\em The category of modules with  good filtrations over a quasi-hereditary algebra has almost split sequences}, Math.~Z.~{\bf 208} (1991), 209--225.
%\bibitem[RH]{RH} S.~Ryom-Hansen, {\em  A $q$-analogue of Kempf's vanishing theorem}, Mosc. Math. J. {\bf 3} (2003), no. 1, 173–187.
%\bibitem[S]{Soe} W.~Soergel, {\em  Character formulas for tilting modules over Kac-Moody algebras},  Represent.~Theory 2 (1998), 432-448. 
%\bibitem[YZ]{YZ} Y.~Yang, G.~Zhao, {\em Frobenii on $E$-theoretical quantum groups}, preprint, {\tt arxiv:2105.14681}, 2021. 
\end{thebibliography}
\end{document}